\documentclass{amsart}

\usepackage [mathscr]{eucal}
\usepackage{amsmath,amsthm,amssymb,amscd}
\usepackage{amsrefs}
\usepackage{epsf}
\usepackage{amsfonts}
\usepackage{dsfont}
\usepackage{latexsym}
\usepackage{mathrsfs}
\usepackage{layout}
\usepackage{bm}

\newtheorem{theorem}{Theorem}[section]
\newtheorem{lemma}[theorem]{Lemma}

\newtheorem{corollary}[theorem]{Corollary}

\newtheorem{definition}[theorem]{Definition}
\numberwithin{equation}{section}

\DeclareMathOperator *{\essosc}{ess\ osc}
\DeclareMathOperator *{\osc}{osc}
\DeclareMathOperator *{\esssup}{ess\ sup}
\DeclareMathOperator *{\essinf}{ess\ inf}
\DeclareMathOperator *{\di}{div} 
\DeclareMathOperator *{\meas}{meas}
\DeclareMathOperator *{\dist}{dist}
\DeclareMathOperator *{\data}{data}
\DeclareMathOperator *{\diam}{diam}
\DeclareMathOperator *{\loc}{loc}
\DeclareMathOperator *{\Lip}{Lip}
\DeclareMathOperator *{\BMO}{BMO}
\DeclareMathOperator *{\BBR}{\mathbb{R}}
\DeclareMathOperator *{\BBC}{\mathbb{C}}
\newcommand{\Z}{{\mathbb Z}}

\begin{document}
\title[Dirichlet parabolic problem with a Carleson condition]{The Dirichlet boundary problem for second order parabolic operators satisfying Carleson condition}

\author{Martin Dindo\v{s}}
\address{School of Mathematics, \\
         The University of Edinburgh and Maxwell Institute of Mathematical Sciences, UK}
\email{M.Dindos@ed.ac.uk}

\author{Sukjung Hwang}
\address{Yonsei University, Korea}
\email{sukjung\_hwang@yonsei.ac.kr}

\date{Received: date / Accepted: date}
\thanks{Both authors  were partially supported by EPSRC EP/J017450/1 grant.}

\begin{abstract} We establish $L^p$, $2\le p\le\infty$ solvability of the Dirichlet boundary value problem for a parabolic equation
$u_t-\mbox{div}(A\nabla u) - \boldsymbol{B}\cdot\nabla u =0$ on time-varying domains with coefficient matrices $A=[a_{ij}]$ and $\boldsymbol{B} = [b_{i}]$ that satisfy a small Carleson condition. The results are sharp in the following sense. For a given value of $1<p<\infty$ there exists operators that satisfy Carleson condition but fail to have $L^p$ solvability of the Dirichlet problem. Thus the assumption of {\it smallness} is {\it sharp}.
Our results complements results of \cite{HL, Rn, Rn2} where solvability of parabolic $L^p$ (for some large $p$) Dirichlet boundary value problem for coefficients that satisfy large Carleson condition was established.
We also give a new (substantially shorter) proof of these results.
\end{abstract}

\maketitle

\section{Introduction}\label{S0:Intro}
This paper is motivated by the known results concerning boundary
value problems for second order divergence form elliptic operators,
when the coefficients satisfy a certain natural, minimal smoothness
condition. To be more specific, consider operators $L= \mbox{div}(A
\nabla) + \boldsymbol{B}\cdot \nabla$ such that $A(X)=[a_{ij}(X)]$ is uniformly elliptic in the
sense that there exist positive constants $\lambda,\,\Lambda$ such that
$$
\lambda|\xi|^2\leq \sum_{i,j} a_{ij}(X)\xi_i\xi_j \le
\Lambda|\xi|^2,
$$
for all $X$ and all $\xi\in \BBR^n$ and under appropriate conditions on the vector $\boldsymbol{B}= [b_{i}]$. We do not assume symmetry
of the matrix $A$. There are a variety of reasons for studying the
non-symmetric situation. These include the connections with
non-divergence form equations, and the broader issue of obtaining
estimates on elliptic measure in the absence of special $L^2$
identities which relate tangential and normal derivatives.

In \cite{KKPT}, the study of nonsymmetric divergence form
operators with bounded measurable coefficients was initiated. In
\cite{KP}, the methods of \cite{KKPT} were used to prove
$A_\infty$ results for the elliptic measure of operators satisfying (a variant of) the Carleson measure
condition. This result was further refined in the paper \cite{DPP} which considered the
$L^p(\partial\Omega)$ Dirichlet problem under the assumption that
\begin{equation}\label{carlI}
\delta(X)^{-1}\left(\mbox{osc}_{B_{\delta(X)/2}(X)}a_{ij}\right)^2
\end{equation}
and
\[
\delta(X)\left(\mbox{sup}_{B_{\delta(X)/2}(X)} b_{i}\right)^2
\]
are the densities of Carleson measures with small Carleson norms.

A recent paper \cite{DPR} has established similar results for the Neuman and Regularity boundary value problems.

The result we present here establish solvability of the $L^p$ Dirichlet boundary value problem
for the parabolic equation 
\begin{equation}
u_t-\mbox{div}(A\nabla u) - \boldsymbol{B}\cdot \nabla u=0
\end{equation}
 with coefficients that satisfy a similar Carleson condition adapted to parabolic settings.
To be specific, if $(X,t)$ is a point in a parabolic domain $\Omega$ (c.f. Definition \ref{D:domain}) (here $X$ denotes the spatial and $t$ the time variable), consider a parabolic distance between points
$$d[(X,t),(Y,\tau)]=(|X-Y|^2+|t-\tau|)^{1/2}.$$
In this metric, we consider the distance function $\delta$ of a point $(X,t)$ to the boundary $\partial\Omega$
$$\delta(X,t)=\inf_{(Y,\tau)\in \partial\Omega}d[(X,t),(Y,\tau)].$$
The parabolic version of the Carleson condition is that
\begin{equation}\label{carlII}
\delta(X,t)^{-1}\left(\mbox{osc}_{B_{\delta(X,t)/2}(X,t)}a_{ij}\right)^2
\end{equation}
and
\[
\delta(X,t) \left(\mbox{sup}_{B_{\delta(X,t)/2}(X,t)} b_{i}\right)^2
\]
are the densities of parabolic Carleson measures with small norms. Here, the ball $B_{\delta(X,t)/2}(X,t)$ is defined using the parabolic metric $d$ defined above. If the coefficients $(a_{ij})$  are time-independent, the condition (\ref{carlII})
becomes the condition (\ref{carlI}) as in the elliptic case.

Operators whose coefficients satisfy  Carleson
condition (\ref{carlII}) arise in the following context. Consider a domain $\Omega$ above a
graph $x_0=\psi(x,t)$, that is the set
$$\{(x_0,x,t):\, x_0>\psi(x,t)\}.$$
Here $X=(x_0,x)$ is the spatial variable ($x_0\in \BBR$, $x\in{\BBR}^{n-1}$) and $t$ denotes the time variable.
We shall assume that $\psi$ is Lipschitz in the variable $x$ and H\"older continuous of order $1/2$ in $t$. Actually, an additional assumption
(a half-derivative in $t$ direction in BMO) is needed, we formulate the condition in detail in the next section.

We consider a mapping $\rho : U \to \Omega$ (c.f. (\ref{mapping})) that maps
the upper half-space $U=\{(x_0,x,t)\in{\BBR}^+\times {\BBR}^{n-1}\times \BBR\}$ into $\Omega$.
If $v_t-\mbox{div}(A\nabla v)-\boldsymbol{B}\cdot \nabla v=0$ in $\Omega$, then $u = v \circ \rho$ will be a solution of a similar parabolic-type
equation $U$. It will be shown that if for example the coefficients of the matrix $A$ are smooth, the corresponding
matrix for the solution $u$ will satisfy a Carleson condition similar to (\ref{carlII}).

Hence, the condition (\ref{carlII}) arises naturally and leads to a question whether together with uniform ellipticity is sufficient for solvability of the
$L^p$ Dirichlet problem for the parabolic equation.

Our main result (Theorem \ref{T:Main}) is a qualitative refinement of \cite{Rn} and more recently \cite{Rn2}, the same way as \cite{DPP} refines \cite{KP} in the elliptic case. We show that the $L^p$ ($p\ge2$) Dirichlet problem for the parabolic equation is solvable, provided the Carleson norm of the coefficients is sufficiently small. As stated in the introduction, this result is {\it sharp} in the sense that the smallness assumption cannot be removed, for each given value of $p$ one can find a (non-symmetric) operator that satisfies all assumptions but has coefficients that only satisfy a sufficiently large Carleson condition for which the $L^p$ solvability fails.

If only large Carleson condition is assumed then one can only conclude as in \cite{Rn, Rn2} solvability of the parabolic problem for some (potentially very large) value of $p>1$ without any refined control on the size of $p$. This is due to the tools used in the proofs of these papers, namely the concepts of $\varepsilon$-aproximability and $A_\infty$ measure. We are able to recover these results as well, thanks to a crucial estimate we establish (Lemma \ref{L:Square2}) and give a significantly simplified argument. 

Our result has connections to other earlier results on the parabolic PDEs. In particular, solvability and $A_\infty$ of the caloric measure under stronger regularity conditions on coefficients and the mapping $\rho : U \to \Omega$ has been studied in Hofmann-Lewis \cite{HL1} and \cite{HL}.

Although our result is motivated by \cite{DPP} where the elliptic result was established, the parabolic problem represented a difficult new challenge where several new ideas were needed. For example, to control the solution in time direction we introduce so-called area function that plays role similar to square function does (in spatial directions) and we also establish relation between these two functions.

We note that previously, the method of layer potentials has been used
to solve parabolic PDE in \cite{Br1}, \cite{Br2} as well as \cite{JM}. Our method does not use layer potentials, instead we rely on a direct method introduced in \cite{DPP} using integration by parts and comparability of square and non-tangential maximal functions. It is not clear whether the rough coefficients we consider allow the use of layer potentials. If so, these results might be extendable to parabolic systems.

The paper is organized as follows. In Section 2, we give definitions and introduce our notation. In section 3 we
state our main results with short proofs. In Section 4 we state some basic (primarily interior) results for the parabolic equation. Estimates for the square function
are contained in Section 5 and finally in Section 6 we estimate the non-tangential maximal function. These two concepts are crucial in our proof. The square function arises naturally, in the process of integration by parts and the non-tangential maximal function is used in formulation of the $L^p$ Dirichlet problem. The fact that these two concepts are comparable in the $L^2$ norm is in the heart of our argument (c.f. Section 7).

{\bf Acknowledgment:} We are very grateful to the anonymous referee for careful reading of this paper which has allowed us to significantly improve the quality of our writing.

\section{Preliminaries}\label{S1:Pre}
\subsection{Admissible parabolic domain $\Omega$}

In the late 70's, Dahlberg \cite{Da} showed that in a Lipschitz domain harmonic measure and surface measure, $d\sigma$, are mutually absolutely continuous,
and furthermore, that the elliptic Dirichlet problem is solvable with data in $L^{2}(d\sigma)$. R. Hunt proposed the problem of finding analogue of
Dalhberg's result for the heat equation in domains whose boundaries are given locally as of functions $\psi(x,t)$ which are Lipschitz in the spatial variable.
It was conjectured at one time that $\psi$ should be $\Lip_{1/2}$ in the time variable, but subsequent counterexamples of Kaufmann and Wu \cite{KW} showed that
this condition does not suffice and that the caloric measure corresponding to the operator $\partial_t-\Delta$ on such domain might not belong to the $A_\infty$ class. Lewis and Murray \cite{LM} made significant progress toward a solution of Hunt's question, by establishing mutual absolute
continuity of caloric measure and a certain parabolic analogue of surface measure in the case that $\psi$ has $1/2$ of a time derivative in $\BMO (\mathbb{R}^n)$
on rectangles, a condition only slightly stronger than $\Lip_{1/2}$.

In this subsection we introduce class of time-varying domains whose boundaries are given locally as functions $\psi(x,t)$, Lipschitz in the spatial variable and
satisfying Lewis-Murray condition in the time variable. At each time $\tau\in\BBR$ the set of points in $\Omega$ with fixed time $t=\tau$, that is
$\Omega_\tau=\{(X,\tau)\in\Omega\}$ will be assumed to be a nonempty bounded Lipschitz domain in $\BBR^n$.  We choose to consider domains that are bounded (in space) since
this most closely corresponds to domains considered the paper \cite{DPP} (for the elliptic equation).
However, our result can be adapted to the case of unbounded domains (in space) (see \cite{HL} which focuses on the unbounded case).

Before we define \lq\lq admissible parabolic domain" we start with few preliminary definitions. If
 $\psi(x,t): \mathbb{R}^{n-1} \times \mathbb{R} \to \mathbb{R}$ is a compactly supported function we define the half time derivative by
\[
D_{1/2}^{t} \psi(x,t) = c_n \int_{\mathbb{R}} \frac{\psi(x,s) - \psi(x,t)}{|s-t|^{3/2}} \,ds
\]
for a properly chosen constant $c_n$ (depending on the dimension $n$). This is equivalent to traditional definition via the Fourier transform.

We shall also need a local version of this definition. If $I\subset \BBR$ is a bounded interval and $\psi(x,t)$ is defined on $\{x\}\times I$ we consider:
\[
D_{1/2}^{t} \psi(x,t) = c_n \int_{I} \frac{\psi(x,s) - \psi(x,t)}{|s-t|^{3/2}} \,ds,\qquad\mbox{for all }t\in I.
\]

We define a parabolic cube in $\mathbb{R}^{n-1} \times \mathbb{R}$,  for a constant $r>0$, as
\begin{equation}\label{D:Q}
Q_{r} (x, t)
= \{ (y, s) \in \mathbb{R}^{n-1}\times\mathbb{R} : |x_i - y_i| < r \ \text{for all } 1 \leq i \leq n-1, \ | t - s |^{1/2} < r \}.
\end{equation}
For a given $f: \mathbb{R}^{n} \to \mathbb{R}$ let,
\[
f_{Q_r} = |Q_r|^{-1} \int_{Q_r} f(x,t) \,dx\,dt.
\]
We say $f \in \BMO (\mathbb{R}^n) $ (this is the parabolic version of the usual BMO space) with the norm $\|f\|_{*}$ if and only if
\[
\|f\|_{*} = \sup_{Q_r} \left\{  \frac{1}{|Q_r|} \int_{Q_r} |f - f_{Q_r} | \,dx\,dt \right\} < \infty.
\]

Again, we also consider a local version of this definition. For a function $f:J\times I\to\BBR$, where $J\subset \BBR^{n-1}$ and $I\subset \BBR$ are closed bounded balls we consider the norm $\|f\|_{*}$ defined as above where the supremum is taken over all parabolic cubes $Q_r$ contained in $J\times I$.\vglue3mm

The following definitions are motivated by the standard definition of a Lipschitz domain.

\begin{definition}
$\Z \subset \BBR^n\times \BBR$ is an $L$-cylinder of diameter $d$ if there
exists a coordinate system $(x_0,x,t)\in \BBR\times\BBR^{n-1}\times \BBR$ obtained from the original coordinate system only by translation in spatial and time variables and rotation in the spatial variables such that
\[
\Z = \{ (x_0,x,t)\; : \; |x|\leq d,\; |t|\leq d^2, \; -(L+1)d \leq x_0 \leq (L+1)d \}
\]
and for $s>0$,
\[
s\Z:=\{(x_0,x,t)\;:\; |x|<sd,\; |t|\leq s^2d^2, \; -(L+1)sd \leq x_0 \leq (L+1)sd \}.
\]
\end{definition}

\begin{definition}\label{D:domain}
$\Omega\subset \BBR^n\times \BBR$ is an {\it admissible parabolic} domain
with `character' $(L,N,C_0)$ if there exists a positive scale $r_0$ such that for any time $\tau\in\BBR$
there are at most $N$ $L$-cylinders $\{{\Z}_j\}_{j=1}^N$ of diameter $d$, with
$\frac{r_0}{C_0}\leq d \leq C_0 r_0$ such that\vglue2mm

\noindent (i) $8{\Z}_j \cap \partial\Omega$ is the graph $\{x_0=\phi_j(x,t)\}$ of a
function $\phi_j$, such that
\begin{equation}
|\phi_j(x,t)- \phi_j(y,s)| \leq L[|x-y|+|t-s|^{1/2}], \qquad
\phi_j(0,0)=0\label{E:L1}
\end{equation}
and
\begin{equation}
\|D^t_{1/2}\phi_j\|_*\le L.\label{E:L2}
\end{equation}
\vglue2mm

\noindent (ii) $\displaystyle \partial\Omega\cap\{|t-\tau|\le d^2\}=\bigcup_j ({\Z}_j \cap \partial\Omega
)$,

\noindent (iii) In the coordinate system $(x_0,x,t)$ of the $L$-cylinder $\Z_j$:
$$\displaystyle{\Z}_j \cap \Omega \supset \left\{
(x_0,x,t)\in\Omega \; : \; |x|<d, \;|t|<d^2\;, \delta(x_0,x,t)=\mathrm{dist}\left( (x_0,x,t),\partial\Omega
\right) \leq \frac{d}{2}\right\}.$$
\end{definition}

Here the distance the the parabolic distance $d[(X,t),(Y,\tau)]=(|X-Y|^2+|t-\tau|)^{1/2}$ introduced in the Section 1.\vglue2mm

\noindent{\it Remark.} It follows from this definition that for each time $\tau\in\BBR$ the time-slice $\Omega_\tau=\Omega\cap\{t=\tau\}$ of an admissible parabolic domain $\Omega\subset \BBR^n\times \BBR$ is a bounded Lipschitz domain in $\BBR^n$ with `character' $(L,N,C_0)$. Due to this fact, the Lipschitz domains $\Omega_{\tau}$ for all $\tau\in\BBR$ have all uniformly bounded diameter (from below and above).

In particular, if ${\mathcal O}\subset \BBR^n$ is a bounded Lipschitz domain, then the parabolic cylinder $\Omega={\mathcal O}\times \BBR$ is an example of a domain satisfying Definition \ref{D:domain}.\vglue2mm

Topologically, any admissible domain $\Omega$ is homeomorphic to the cylinder $\Omega_{\tau}\times\BBR$ for any $\tau\in\BBR$. This is due to the fact that any two sets $\Omega_{\tau_1}$, $\Omega_{\tau_2}$ with $|\tau_1-\tau_2|<(\frac{r_0}{C_0})^2$ are topologically equivalent. Hence any two $\Omega_{\tau_1}$, $\Omega_{\tau_2}$ are homeomorphic. From this the existence of homeomorphism $\Omega\to \Omega_{\tau}\times\BBR$ follows.

\begin{definition}\label{D:measure}
Let $\Omega\subset \BBR^n\times \BBR$ be an {\it admissible parabolic} domain
with `character' $(L,N,C_0)$. Consider the following measure $\sigma$ on $\partial\Omega$. For $A\subset \partial\Omega$ let
\begin{equation}\label{E:sigma}
\sigma(A)=\int_{-\infty}^\infty {\mathcal H}^{n-1}\left(A\cap\{(X,t)\in\partial\Omega\}\right)dt.
\end{equation}
Here ${\mathcal H}^{n-1}$ is the $n-1$ dimensional Hausdorff measure on the Lipschitz boundary $\partial\Omega_t=\{(X,t)\in\partial\Omega\}$.
\end{definition}

We are going to consider solvability of the $L^p$ Dirichlet boundary value problem with respect to the measure $\sigma$. Note that under our assumption this measure
might not be comparable to the usual surface measure on $\partial\Omega$. This is due to the fact that in the $t$-direction the functions $\phi_j$ from the Definition \ref{D:domain}
are only half-Lipschitz and hence the standard surface measure might not be locally finite.

Our definition assures that for any $A\subset \Z_j\cap\partial\Omega$, where  $\Z_j$ is an $L$-cylinder we have
\begin{equation}\label{E:comp}
\sigma(A) \approx {\mathcal H}^n\left((x,t):\,(\phi_j(x,t),x,t)\in A\}\right),
\end{equation}
where the actual constants in \eqref{E:comp} by which these measures are comparable only depend on the $L$ of the `character' $(L,N,C_0)$ of domain $\Omega$.\vglue2mm

If $\Omega$ has smoother boundary, such as Lipschitz (in all variables) or better, then the measure $\sigma$ is comparable to the usual $n$-dimensional Hausdorff measure ${\mathcal H}^{n}$. In particular, this holds for the parabolic cylinder $\Omega={\mathcal O}\times \BBR$ mentioned above.

\subsection{Pullback transformation and Carleson condition}\label{S11:PC}

In this paper, we consider the parabolic differential equation
\begin{equation}\label{E:v}\begin{cases}
v_t = \di (A^{v} \nabla v)  + \boldsymbol{B}^{v} \cdot \nabla v & \text{in } \Omega, \\
v   = f^{v}                & \text{on } \partial \Omega
\end{cases}\end{equation}
where $A^{v}= [a^{v}_{ij}(X, t)]$ is an $n\times n$ matrix satisfying the uniform ellipticity condition and $\boldsymbol{B}^{v} = [b^{v}_{i}(X,t)]$ is a locally bounded $1 \times n$ vector with $X \in \mathbb{R}^{n}$, $t\in \mathbb{R}$, that is, there exists positive constants $\lambda^{v}$ and $\Lambda^{v}$ such that
\begin{equation}\label{E:UEBv}
\lambda^{v} |\xi|^2 \leq \sum_{i,j} a^{v}_{ij} \xi_i \xi_j \leq \Lambda^{v} |\xi|^2
\end{equation}
for all $\xi \in \mathbb{R}^{n}$. We work on \lq\lq admissible" domains $\Omega$ introduced above.\vglue2mm

Here and throughout the paper we will consistently use the notation denoting $\nabla v$ the gradient in the spatial variables, $v_t$ or $\partial_t v$ the gradient in the time variable and $Dv=(\nabla v, \partial_t v)$ the full gradient of $v$.\vglue2mm


We return to the pullback transformation. For simplicity (to avoid getting bogged down in technical details connected with localization) consider for the moment that
\begin{equation}\label{E:Omega}
\Omega = \{ (x_0, x, t) \in \mathbb{R}\times \mathbb{R}^{n-1} \times \mathbb{R} : x_0 > \psi(x,t) \}
\end{equation}
where $\psi(x,t): \mathbb{R}^{n-1} \times \mathbb{R} \to \mathbb{R}$ has compact support and satisfies condition (i) of the Definition \ref{D:domain}.

Our strategy to show the $L^2$ solvability of the PDE \eqref{E:v} is to take a pullback transformation $\rho:U\to\Omega$ and consider a new parabolic PDE
on the upper half-space
\begin{equation}\label{E:U}
U = \{ (x_0, x, t) : x_0 > 0, \ x \in \mathbb{R}^{n-1}, \ t \in \mathbb{R} \}
\end{equation}
obtained from the original PDE via the pullback.
To motivate the choice of our mapping $\rho$ consider first the obvious map $\tilde{\rho}: U \to \Omega$ defined by
\[
\tilde{\rho} (x_0, x, t) = (x_0 + \psi(x,t), x, t), \quad x\in\mathbb{R}^{n-1}, t\in \mathbb{R}.
\]
The new PDE for $u = v \circ \tilde{\rho}$ will yield an additional drift (first order) term
\[
\psi_{t} (X, t) u_{x_0} (X, t).
\]
Observe that $\psi_t$ might not defined everywhere because $\psi$ lacks the regularity in the $t$-variable (and hence $\boldsymbol{B}$ might be unbounded). A similar issue arises with the second-order coefficients; any regularity of the original coefficients $A^{v}$  will be lost after the pullback due to presence of $\psi_x$ which is only bounded.

To overcome this difficulty, we consider a mapping $\rho : U \to \Omega$ (c.f. \cite{HL}) in the setting of parabolic equations defined by
\begin{equation}
\rho (x_0, x, t) = (x_0 + P_{\gamma x_0}\psi(x,t), x, t).\label{mapping}
\end{equation}
See also \cite{Ne} for other uses of this map.
To define $P_{\gamma x_0}$, consider a non-negative function $P(x,t) \in C_{0}^{\infty}(Q_{1}(0,0))$, for $(x,t) \in \mathbb{R}^{n-1}\times\mathbb{R}$, and set
\[
P_{\lambda} (x,t) \equiv \lambda^{-(n+1)} P\left( \frac{x}{\lambda}, \frac{t}{\lambda^2} \right)
\]
and
\[
P_{\lambda} \psi(x,t) \equiv \int_{\mathbb{R}^{n-1}\times \mathbb{R}} P_{\lambda} (x-y, t-s) \psi(y,s) \,dy\,ds.
\]
Then $\rho$ satisfies
\[
\lim_{(y_0, y, s) \to (0, x, t)} P_{\gamma y_0} \psi(y,s) = \psi(x, t)
\]
and extends continuously to $\rho:\overline{U}\to \overline{\Omega}$. As follows from the discussion above the usual
surface measure on $\partial U$ is comparable with the measure $\sigma$ defined by \eqref{E:sigma} on $\partial\Omega$.

Suppose that $u = v \circ \rho$ and $f = f^{v} \circ \rho$. Then the PDE \eqref{E:v} transforms to a new PDE for the variable $u$
\begin{equation}\label{E:u}\begin{cases}
u_t = \di (A \nabla u) + \boldsymbol{B}\cdot \nabla u & \text{in } U, \\
u   = f                & \text{on } \partial U
\end{cases}\end{equation}
where $A = [a_{ij}(X, t)]$, $\boldsymbol{B} = [b_i (X, t)]$ are a $(n\times n)$ and $(1\times n)$ matrices. 
The precise relations between the original coefficients $A^v$ and $\boldsymbol{B}^v$ and the new coefficients $A$ and $\boldsymbol{B}$ for $u$ are worked out in \cite[pp. 448]{Rn2}, we refer the reader there for the details. 

We want to find properties of the coefficients $A$ and $\boldsymbol{B}$ of the parabolic equation \eqref{E:u}. We note that if the constant $\gamma > 0$ is chosen small enough then for $(x,t)\in \mathbb{R}^{n-1}\times \mathbb{R}$,
\[
\frac{1}{2} \leq 1 + \partial_{x_0} P_{\gamma x_0}\psi(x,t) \leq \frac{3}{2},
\]
and then the coefficients $a_{ij}, b_i : U \to \mathbb{R}$ are Lebesgue measurable and $A$ satisfies the standard ellipticity condition, since the original matrix $A^v$ did. That is there exist constants $\lambda$ and $\Lambda$ such that
\begin{equation}\label{E:UEB}
\lambda |\xi|^2 \leq \sum_{i j} a_{ij} \xi_i \xi_j \leq \Lambda |\xi|^2
\end{equation}
for any $\xi \in \mathbb{R}^n$.

\begin{definition} Let $\Omega$ be an admissible parabolic domain from Definition \ref{D:domain}.
For $(X,t) \in{\BBR}^{n}\times \BBR$ and $(Y,s)\in\partial\Omega$ and $r>0$ we write:
\begin{align*}
B_r(X,t) &=\{(Z,\tau)\in{\BBR}^{n}\times \BBR:\, d[(X,t),(Z,\tau)]<r\}\\
\Delta_r(Y,s) &= \partial \Omega\cap B_r(Y,s),\,\,\,\qquad T(\Delta_r) = \Omega\cap B_r(Y,s).
\end{align*}
\end{definition}

Here $d$ is the parabolic distance.

\begin{definition}\label{D:carl}
Let $T(\Delta_r)$ be the Carleson region associated to a surface
ball $\Delta_r$ in $\partial\Omega$, as defined above. A measure $\mu:\Omega\to\BBR^+$  is said to be
Carleson if there exists a constant $C=C(r_0)$ such that for all
$r\le r_0$ and all surface balls $\Delta_r$
\[\mu(T(\Delta_r))\le C \sigma (\Delta_r).\]
The best possible constant $C$ will be called the Carleson norm and shall be denoted by $\|\mu\|_{C,r_0}$. We write $\mu \in \mathcal C$.  If
$\displaystyle\lim_{r_0\to 0} \|\mu\|_{C,r_0}=0$,  we say that the
measure $\mu$ satisfies the vanishing Carleson condition and write $\mu \in \mathcal C_V$. 
Occasionally, for brevity we drop $r_0$ and just write  $\|\mu\|_{C}$ if the maximal radius of ball over which we calculate the Carleson norm is clear from the context. 
\end{definition}

When $\partial\Omega$ is locally given as a graph of a function $x_0=\psi(x,t)$ in the coordinate system $(x_0,x,t)$
and $\mu$ is a measure supported on $\{x_0>\psi(x,t)\}$ we can reformulate the Carleson condition locally using the parabolic cubes $Q_r$
and corresponding Carleson regions $T(Q_r)$ where
\begin{align*}
Q_{r} (y, s)
&= \{ (x, t) \in \mathbb{R}^{n-1}\times\mathbb{R} : |x_i - y_i| < r \ \text{for all } 1 \leq i \leq n-1, \ | t - s |^{1/2} < r \}\\
T(Q_r)&=\{(x_0,x, t) \in \mathbb{R}\times\mathbb{R}^{n-1}\times\mathbb{R} : \psi(x,t)<x_0<\psi(x,t)+r,\, (x,t)\in Q_{r} (y, s)\}.
\end{align*}

The Carleson condition becomes
\[\mu(T(Q_r))\le C|Q_r|=Cr^{n+1}.\]

We remark, that the corresponding Carleson norm will not be equal to the one from Definition \ref{D:carl} but these norms will be comparable. Hence the notion of
vanishing Carleson norm does not change if we take this as the definition of the Carleson norm instead of Definition \ref{D:carl}.

We also want to define $Q_r(Y,s)$ for $(Y,s)\in \mathbb{R}^{n-1}\times\mathbb{R}$, this ise defined as $Q_r(y,s)$ where $Y=(y_0,y)$.\vglue2mm

Observe also, that the function $\delta(X,t):=\inf_{(Y,\tau)\in \partial\Omega}d[(X,t),(Y,\tau)]$ that is measuring the distance of a point $(X,t)=(x_0,x,t)\in\Omega$ to the boundary $\partial\Omega$ is comparable to $x_0-\psi(x,t)$ which in turn is comparable to $[\rho^{-1}(X,t)]_{x_0}$ (the first component of the inverse map $\rho^{-1}$).

\vglue2mm


We now return to the pullback map $\rho:U\to\Omega$. We first note the Lemma A of \cite{HL} implying further structure of the transformed coefficients.

\begin{lemma}\label{L:A}
Let $\sigma$, $\theta$ be nonnegative integers and $\phi = (\phi_1, \ldots, \phi_{n-1})$, a multi-index, with $l=\sigma + |\phi| + \theta$. If $\psi$ satisfies that for all $x, y \in \mathbb{R}^{n-1}$, $t,s\in\mathbb{R}$ and for some positive constants $L_1, L_2 < \infty$
\[
|\psi(x,t) - \psi(y, s)| \leq L_1 \left(|x-y| + |t-s|^{1/2}\right)
\]
and
\[
\|D_{1/2}^{t} \psi\|_{*} \leq L_{2},
\]
then the measure $\nu$ defined at $(x_0, x, t)$ by
\[
d \nu = \left( \frac{\partial^{l} P_{\gamma x_0} \psi}{\partial x_{0}^{\sigma} \partial x^{\phi} \partial t^{\theta}} \right)^2 x_{0}^{2l + 2\theta - 3} \, dx \,dt \,dx_0
\]
is a Carleson measure whenever either $\sigma + \theta \geq 1$ or $|\phi| \geq 2$, with
\[
\nu \left[ (0, r) \times Q_{r}(x,t) \right] \leq c \left| Q_{r} (x,t) \right|.
\]
Moreover, if $l \geq 1$, then at $(x_0, x, t)$
\[
\left| \frac{\partial^{l} P_{\gamma x_0} \psi}{\partial x_{0}^{\sigma} \partial x^{\phi} \partial t^{\theta}} \right|
\leq c' (L_1 + L_2)  x_{0}^{1-l-\theta}
\]
where $c' = c'(n)$ and $c=c(L_1, L_2, \gamma, l, n) \geq 1$.
\end{lemma}

The drift term $\boldsymbol{B}$ from the pullback transformation in \eqref{E:u} includes
\[
\frac{\partial}{\partial t} P_{\gamma x_0} \psi u_{x_0} .
\]
From Lemma~\ref{L:A} with $\sigma = |\phi| = 0$, $\theta = 1$, we see that
\[
x_0 \left[ \frac{\partial}{\partial t} P_{\gamma x_0} \psi (x,t) \right]^2 \,dX\,dt
\]
is a Carleson measure on $U$. Thus it is natural to expect that $\boldsymbol{B}$ will satisfy
\begin{equation}\label{E:B1}
x_0 |\boldsymbol{B}| (X,t) \leq \Lambda_B < C^{1/2}_{\epsilon}
\end{equation}
and
\begin{equation}\label{E:B2}
d\mu_1 (X,t) = x_0 |\boldsymbol{B}|^2(X,t) \,dX\,dt
\end{equation}
is a Carleson measure on $U$ with Carleson constant $C_{\epsilon}$. Indeed, this is the case provided the original vector vector $\boldsymbol{B}^{v}$ satisfies
the assumption that
\begin{equation}\label{carlB}
d\mu(X,t) = \delta(X,t) \left[\sup_{B_{\delta(X,t)/2}(X,t)}|\boldsymbol{B}^v|\right]^2\,dX\,dt
\end{equation}
is the density of Carleson measure in $\Omega$. Here $C_\epsilon=\|\mu_1\|_{C}$ depends on the Lipschitz constant $L$ (Definition \ref{D:domain}) and the Carleson norm of (\ref{carlB}).

Similarly, for the matrix $A$ we apply Lemma~\ref{L:A} with either $\sigma = 1$, $\phi =1$, $\theta = 0$ and $l=2$ or  $\sigma =\theta = 0$, $\phi =2$, and $l=2$ for $\nabla A$. For $A_{t}$, we take $\sigma = 0$, $\phi=\theta = 1$, and $l =2$. It follows using the calculation in \cite{Rn2} that $A$ will satisfy
\begin{equation}\label{E:A1}
(x_0 |\nabla A| + x_0^2 |A_t|) (X,t) < C^{1/2}_{\epsilon}
\end{equation}
for almost everywhere $(X,t) \in U$ and
\begin{equation}\label{E:A2}
d\mu_2 (X,t) = (x_0 |\nabla A|^2 + x_0^3 |A_t|^2)(X,t) \,dX\,dt
\end{equation}
is a Carleson measure on $U$ with the Carleson norm $C_{\epsilon}$,
provided the original matrix $(A^v)$ satisfies that
\begin{equation}\label{carlIII}
\begin{split}
&d\mu (X,t) = \\&\left(\delta(X,t) \left[\sup_{B_{\delta(X,t)/2}(X,t)}|\nabla A^v|\right]^2+\delta(X,t)^3 \left[\sup_{B_{\delta(X,t)/2}(X,t)}|\partial_t A^v|\right]^2\right)\,dX\,dt
\end{split}
\end{equation}
is the density of Carleson measure in $\Omega$. We note that if both $\|\mu\|_{C}$ and $L$ are small, then so is Carleson norm $C_{\epsilon}=\|\mu_1\|_{C}$ of the matrix $A$.

Observe that the condition (\ref{carlIII}) is slightly stronger than the condition (\ref{carlII}) we have claimed to assume in the introduction. We shall replace
the condition (\ref{carlIII}) by the weaker condition (\ref{carlII}) via perturbation results of \cite{Sw}, the details are in the following section.

\subsection{Admissible parabolic domains revisited.} \label{S:alld} We now return to the parabolic domains considered in Definition \ref{D:domain}. As follows from this definition, we can consider locally on each $L$-cylinder $\Z_j$ the pullback map $\rho_j$ defined as above since the boundary $\partial\Omega$ on $\Z_j$ is given as a graph of a function $\phi_j$.

We adapt results from the paper \cite{BZ}. Firstly, by Proposition 2.1 \cite{BZ} (the statement is for a bounded domain but it adapts to our case of an unbounded
domain in time direction), there exists a neighborhood $V$ of $\partial\Omega$ and smooth function $G:V\to {\mathbb S}^{n}$ such that for each $(X,t)\in U$
the unit vector $G(X,t)$ is in \lq good' direction. Here ${\mathbb S}^{n}\subset {\BBR}^{n+1}$ is the $n$-dimensional sphere. What that means is that with respect to a small ball around $(X,t)$ the boundary $\partial\Omega$ looks like a graph of a function with $x_0$ coordinate in the direction $G(X,t)$ (c.f. (i) of Definition \ref{D:domain}). Moreover, in our case the last (time component) of vector $G(X,t)$ vanishes.

Secondly, the concept of \lq\lq proper generalized distance" \cite[Proposition 3.1]{BZ} can be adapted to our setting. The function $\delta(X,t)$ measuring parabolic distance of a point $(X,t)\in\Omega$ to the boundary $\partial\Omega$ has been defined earlier. We claim that there exists a function $\ell\in C(\overline{\Omega})\cap C^\infty(\Omega)$ such that
$$\frac1K\le \frac{\ell(X,t)}{\delta(X,t)}\le K,$$
$$\nabla \ell(X,t)\ne 0,\qquad\text{for all $(X,t)$ in a neighborhood of }\partial\Omega,\,(X,t)\notin\partial\Omega$$
$$|\ell(X,t)-\ell(Y,s)|\le K[|X-Y|^2+|t-s|]^{1/2}.$$
Here $K\ge 1$ only depends on the character $(L,N,C_0)$ of the domain $\Omega$. It follows that $\ell$ can be used in place of the function $\delta$, but has an additional interior regularity. We construct $\ell$ slightly differently than in Proposition 3.1 of \cite{BZ}. On each $L$-cylinder $\Z$ as in Proposition \ref{D:domain} we have a map $\rho$ mapping neighborhood of $0\in U$ to a neighborhood of a boundary point in $\Omega$. For a point $(X,t)\in\Omega$ we define $\ell(X,t)=[\rho^{-1}(X,t)]_{x_0}$ where $[\cdot]_{x_0}$ denotes the first component of the vector in $U$.
This is equivalent to solving the following implicit equation:
$$x_0=\ell(X,t)+\int_{Q_1(0,0)}P(y,s)\phi(x-\gamma\ell(X,t)y,t-\gamma^2\ell^2(X,t)s)dy\,ds.$$
Here, $(X,t)=(x_0,x,t)$, $P$ is the function defined below (\ref{mapping}) and $\phi$ is the function defining $\partial\Omega$ as a graph on $\Z$.
This is essentially how $\ell$ is defined in Proposition 3.1 of \cite{BZ}, our modification takes into account the parabolic scaling of the metric $d$ in the time variable. We now construct a global function $\ell$ via gluing these functions on each coordinate chart via partition of unity on a neighborhood of $U$.
This will preserve
$$\nabla \ell(X,t)\ne 0,\qquad\text{for all $(X,t)$ in a neighborhood of }\partial\Omega,\,(X,t)\notin\partial\Omega$$
at least when the constant $L$ in the character of our domain $\Omega$ is small, since that ensures that overlapping coordinate charts are almost parallel.

We now have the result of Theorem 5.1 of \cite{BZ}. There exists $\epsilon_0>0$ such that for all $0<\epsilon\le\epsilon_0$ then
$$\Omega^\epsilon=\{(X,t)\in{\BBR}^{n+1}:\, \ell(X,t)>\epsilon\}$$
is a domain of class $C^\infty$ and there is a homeomorphism $f^\epsilon:\overline{\Omega}\to\overline{\Omega^\epsilon}$ such that $f_\epsilon(\partial\Omega)=\partial\Omega^\epsilon$ and $f_\epsilon:{\Omega}\to{\Omega^\epsilon}$ is a $C^\infty$ diffeomorphism.

In addition, if $\Omega_\tau$ and $\Omega^\epsilon_\tau$ denote the time slices of $\Omega$, $\Omega^\epsilon$ for a fixed time $t=\tau$ then $f_\epsilon:\overline{\Omega_\tau}=\overline{\Omega^\epsilon_\tau}$ is a bi-Lipschitz homeomorphism with Lipschitz constant independent of $\epsilon$ and $\tau$
and depending only on the $L$ in the character $(L,N,C_0)$ of the domain $\Omega$. In particular, this Lipschitz constant is small if $L$ is small.

\subsection{Parabolic Non-tangential cones and related functions} \label{S12:Not}

We proceed with the definition of parabolic non-tangential cones. We define the cones in a (local) coordinate system where $\Omega=\{(x_0,x,t):\,x_0>\psi(x,t)\}$. In particular this also applies to the upper half-space $U=\{(x_0,x,t), x_0>0\}$. We note here, that a different choice of coordinates (naturally) leads to different sets of cones, but as we shall establish the particular choice of non-tangential cones is not important as it only changes constants in the estimates for the area, square and non-tangential maximal functions defined using these cones. However the norms defined using different sets of non-tangential cones are comparable.

For a constant $a>0$, we define the parabolic non-tangential cone at a point $(x_0,x,t)\in\partial\Omega$ as follows
\begin{equation}\label{D:Gammag}
\Gamma_{a}(x_0, x, t) = \left\{(y_0, y, s)\in \Omega : |y - x| + |s-t|^{1/2} < a(y_0 - x_0), \  y_0 > x_0 \right\}.
\end{equation}
We occasionally truncate the cone $\Gamma$ at the height $r$
\begin{equation}\begin{split}\label{D:Gammagr}
&\Gamma_{a}^{r}(x_0, x, t) =\\ &\left\{(y_0, y, s)\in \Omega : |y - x| + |s-t|^{1/2} < a(y_0 - x_0), \  x_0 < y_0 < x_0 + r  \right\}.
\end{split}
\end{equation}

When working on the upper half space (domain $U$), $(0, x, t)$ is the boundary point of $\partial U$. In this case we shorten the notation and write
\begin{equation}\label{D:Gamma}
\Gamma_{a}(x,t) \qquad\mbox{instead of }\qquad \Gamma_{a}(0, x, t)
\end{equation}
and
\begin{equation}\label{D:Gammar}
\Gamma_{a}^{r}(x,t) \qquad\mbox{instead of }\qquad \Gamma^r_{a}(0, x, t).
\end{equation}

Observe that the slice of the cone $\Gamma_{a}(x_0, x, t)$ at a fixed height $h$ is the set
$$\{(y,s):\,(x_0+h,y,s)\in\Gamma_a(x_0,x,t)\}$$
which contains and is contained in a parabolic box $Q_{s}(x,t)$ of radius $s$ comparable to $h$, that is
for some constants $c_1, c_2$ depending only on the dimension $n$ and $a$ we have
\[
Q_{c_1 h}(x,t) \subset \{(y,s):\,(x_0+h,y,s)\in\Gamma_a(x_0,x,t)\} \subset Q_{c_2 h}(x,t).
\]

For a function $u : \Omega \rightarrow \mathbb{R}$, {\it the nontangential maximal function} $\partial\Omega\to \BBR$ and its truncated version at a height $r$ are defined as
\begin{equation}\label{D:NTan}\begin{split}
N_{a}(u)(x_0,x,t) &= \sup_{(y_0,y,s)\in \Gamma_{a}(x_0,x,t)} \left|u(y_0 , y, s)\right|, \\
N_{a}^{r}(u)(x_0,x,t) &= \sup_{(y_0,y,s)\in \Gamma_{a}^{r}(x_0,x,t)} \left|u(y_0 , y, s)\right|\quad\mbox{for }(x_0,x,t)\in\partial\Omega.
\end{split}\end{equation}

Now we define {\it the square function} $\partial\Omega\to \BBR$ (and its truncated version) asssuming $u$ has a locally integrable distributional gradient by
\begin{equation}\begin{split}\label{D:Square}
S_{a}(u)(x_0,x,t) &= \left(\int_{\Gamma_{a}(x_0,x,t)} (y_{0}-x_0)^{-n} |\nabla u|^{2}(y_0, y, s) \,dy_0\,dy\,ds\right)^{1/2}, \\
S_{a}^{r}(u)(x_0,x,t) &= \left(\int_{\Gamma_{a}^{r}(x_0,x,t)} (y_{0}-x_0)^{-n} |\nabla u|^{2}(y_0, y, s) \,dy_0\,dy\,ds\right)^{1/2}.
\end{split}\end{equation}
Observe that on the domain $U=\{(x_0,x,t):\,x_0>0\}$
\[
\|S_a (u)\|^{2}_{L^2(\partial U)} \approx \int_{U} y_{0} |\nabla u|^{2}(y_0, y, s) \,dy_0\,dy\,ds.
\]
where the implied constant depend on the aperture of the non-tangential cone.
\vglue2mm

Finally, we shall also need an object we call {\it the area function} $\partial\Omega\to \BBR$ defined by
\begin{equation}\label{D:Area}\begin{split}
A_{a} (u) (x_0,x,t) &= \left(\int_{\Gamma_{a}(x_0,x,t)} (y_{0}-x_0)^{-n+2} |u_{t}|^{2}(y_0, y, s) \,dy_0\,dy\,ds\right)^{1/2}, \\
A_{a}^{r} (u) (x_0,x,t) &= \left(\int_{\Gamma_{a}^{r}(x_0,x,t)} (y_{0}-x_0)^{-n+2} |u_{t}|^{2}(y_0, y, s) \,dy_0\,dy\,ds\right)^{1/2}.
\end{split}\end{equation}
Observe that the domain $U=\{(x_0,x,t):\,x_0>0\}$
\[
\|A_{a} (u)\|^{2}_{L^2 (\partial U)} \approx \int_{U} y_{0}^3 |u_{t}|^{2}(y_0, y, s) \,dy_0\,dy\,ds.
\]
Clearly, the square function can be used to control oscillation of a solution $u$ in the spatial directions
and similarly, the area function controls the solution in the time variable. Hence these two functions together allow us to control the solution $u$
in all variables.  We also note that we use the name area function because there is an obvious connection with \lq\lq area function" defined previously for elliptic PDEs which contains the term $\nabla^2 u$. In our case, the parabolic PDE $u$ satisfies implies that $|u_t|^2$ and $|\nabla^2 u|^2$ are closely related.

\subsection{$L^p$ Solvability of the Dirichlet boundary value problem}

\begin{definition} Let $1<p\le\infty$  and $\Omega$ be an admissible parabolic domain from the Definition \ref{D:domain}.
Consider the parabolic Dirichlet boundary value problem
\begin{equation}\label{E:para}\begin{cases}
v_t = \di (A \nabla v) + \boldsymbol{B}\cdot \nabla v & \text{in } \Omega, \\
v   = f \in L^p(\partial\Omega,d\sigma)              & \text{on } \partial \Omega,\\
N(v) \in L^p(\partial\Omega,d\sigma).
\end{cases}\end{equation}
where the matrix $A= [a_{ij}(X, t)]$ satisfies the uniform ellipticity condition, the vector $\boldsymbol{B}=[b_{i}]$ is locally bounded and $\sigma$ is the measure supported on $\partial\Omega$ defined by \eqref{E:sigma}.

We say that Dirichlet problem with data in $L^p(\partial \Omega, d\sigma)$
is solvable if the (unique) solution $u$ with continuous boundary data $f\in C_0(\partial\Omega)$, where 
$$C_0(\partial\Omega)=\{f:\partial\Omega\to\mathbb R:\, f\in C(\partial\Omega)\quad\&\quad |f|\to 0\mbox{ uniformly as }|t|\to\infty\}$$
satisfies the estimate
\begin{equation}\label{E:estim}
\|N(v)\|_{L^p(\partial \Omega, d\sigma)} \lesssim \|f\|_{L^p(\partial \Omega, d\sigma)}.
\end{equation}
The implied constant depends only the operator,
$p$, and the the triple $(L,N,C_0)$ of Definition \ref{D:domain}.
\end{definition}

{\it Remark.} It is well-known that the parabolic PDE (\ref{E:para}) with boundary data on $C_0(\partial\Omega)$ is uniquely solvable in the class $C_0(\overline\Omega)$. (There are continuous functions on $\overline\Omega$ with uniform decay to $0$ at $t\to\pm\infty$). This can be established by considering
approximation of bounded measurable coefficients of matrix $A$ by a sequence of smooth matrices $A_j$ and then taking the limit $j\to\infty$. This limit will
exists in $L^\infty(\Omega)\cap W^{1,2}_{loc}(\Omega)$ by the maximum principle (Lemma \ref{L:MP}) and the classical existence theory in $L^\infty(\mathbb R,W^{1,2}(\Omega))$. Uniqueness (for boundary data in $C_0(\partial\Omega)$) is also a consequence of the maximum principle.

If $p<\infty$, the space $C_0(\partial \Omega)$ is dense in $L^p(\partial \Omega, d\sigma)$. It follows that if the estimate
\[\|N(u)\|_{L^p(\partial \Omega, d\sigma)} \lesssim \|f\|_{L^p(\partial \Omega, d\sigma)}\]
holds for all continuous data, then for any $f\in L^p(\partial \Omega, d\sigma)$ there exists a solution $u$ to the equation (\ref{E:para}) such that \eqref{E:estim}
holds (by the continuous extension of the solution operator from $C_0(\partial \Omega)$ to $L^p(\partial \Omega,d\sigma)$).
Moreover, it can be shown that
$$u(X,t)=\lim_{(Y,s)\in\Gamma(X,t),\,(Y,s)\to (X,t)}u(Y,s),\qquad\mbox{for a.e. }(X,t)\in\partial\Omega.$$
\vglue2mm

{\it Remark 2.} The boundary value problem (\ref{E:para}) is defined on a domain unbounded in time (on both ends). However, once solvability of (\ref{E:para})
is established, the solvability of the following initial value problem also holds.
\begin{equation}\label{E:para2}\begin{cases}
v_t = \di (A \nabla v) + \boldsymbol{B}\cdot \nabla v & \text{in } \Omega \text{ for all $t>0$}, \\
v   = f \in L^p              & \text{on } \partial \Omega\cap\{t>0\},\\
v(X,0)=0&\text{on }\Omega\cap \{t=0\},\\
N(v) \in L^p(\partial\Omega\cap\{t>0\}).
\end{cases}\end{equation}
Indeed, if ${\mathcal O}=\Omega\cap \{t=0\}$ we might just consider $\Omega\cap \{t\le 0\}={\mathcal O}\times (-\infty,0]$. If we extend $f$ defined on $\partial \Omega\cap\{t>0\}$ onto whole $\Omega$ by setting $f=0$ on $\partial{\mathcal O}\times (-\infty,0]$ then the solution of (\ref{E:para}) restricted to
$\Omega\cap \{t\ge 0\}$ solves (\ref{E:para2}) since $u=0$ for $t\le 0$ and therefore $u=0$ at $t=0$.

A similar consideration also establishes solvability on a time interval $(-\infty,T]$, $T<\infty$ by considering an extension of $f$ by zero for $t>T$.\vglue2mm

\noindent {\it Remark 3.} (Parabolic measure). Since the equation (\ref{E:para}) has a unique continuous solution there exists a a measure $\omega^{(X,t)}$ such that
$$u(X,t)=\int_{\partial \Omega} f(Y,s) d\omega^{(X,t)}(Y,s)$$
for all continuous data called the parabolic measure. Under the assumption of Definition \ref{D:domain} and the drift term $\boldsymbol{B}$ having small Carleson norm this measure is doubling (c.f. \cite[Lemma 3.12 and 3.14]{HL}).
In this case, the $L^p$ solvability of the Dirichlet boundary value problem for some $p<\infty$ is equivalent to the parabolic measure $\omega$ being $A_\infty$
with respect to the measure $\sigma$ on the surface $\partial\Omega$.

\section{The Main Results}\label{S2:MainR}

\begin{theorem}\label{T:Main} Let $\Omega$ be a domain as in the Definition \ref{D:domain} with character $(L,N,C_0)$. Let $A=[a_{ij}]$ be a matrix with bounded
measurable coefficients defined on $\Omega$ satisfying the uniform ellipticity and boundedness with constants $\lambda$ and $\Lambda$ and $\boldsymbol{B} = [b_{i}]$ be a vector with  measurable coefficients defined on $\Omega$. In addition, assume that
\begin{equation}\label{E:carl4}
d\mu=\left[\delta(X,t)^{-1} \sup_{1\le i,j\le n}\left( \osc_{B_{\delta(X,t)/2}(X,t)} a_{ij} \right)^2 + \delta(X,t)\sup_{B_{\delta(X,t)/2}(X,t)} \left| \boldsymbol{B} \right|^{2} \right] \,dX\,dt
\end{equation}
is the density of a Carleson measure on $\Omega$ with Carleson norm $\|\mu\|_C$. Then there exists $\varepsilon>0$ such that if for some $r_0>0$ we have $\max\{L^2,\|\mu\|_{C,r_0}\}<\varepsilon$ then the $L^p$ boundary value problem
\begin{equation}\label{E:para3}\begin{cases}
v_t = \di (A \nabla v) + \boldsymbol{B}\cdot \nabla v & \text{in } \Omega, \\
v   = f \in L^p(\partial\Omega,d\sigma)              & \text{on } \partial \Omega,\\
N(v) \in L^p(\partial\Omega,d\sigma),
\end{cases}\end{equation}
is solvable for all $2\leq p < \infty$. Moreover, the estimate
\begin{equation}\label{E:estim3}
\|N(v)\|_{L^p(\partial \Omega, d\sigma)} \le C_p \|f\|_{L^p(\partial \Omega, d\sigma)},
\end{equation}
holds with $C_p=C_p(L,N,C_0,r_0,\lambda,\Lambda)$. It also follows that the parabolic measure of the operator $\partial_t-\di ({A} \nabla \cdot) - \boldsymbol{B}\cdot \nabla $ is doubling and belongs to $B_2(d\sigma)\subset A_\infty(d\sigma)$.
\end{theorem}

Instead of \eqref{E:carl4} we can state the result using alternative assumptions. These are as in Theorem 2.13 of \cite{HL} (without few unnecessary extra technical conditions).

\begin{theorem}\label{T:Main2} Let $\Omega$ be a domain as in the Definition \ref{D:domain} with character $(L,N,C_0)$. Let $A=[a_{ij}]$ be a matrix with bounded
measurable coefficients defined on $\Omega$ satisfying the uniform ellipticity and boundedness with constants $\lambda$ and $\Lambda$ and $\boldsymbol{B}=[b_i]$ be a vector with measurable coefficients defined on $\Omega$. In addition, assume that
\begin{equation}\label{E:carl4a}
d\mu=\left(\delta(X,t) |\nabla A|^2+ \delta^3(X,t) |\partial_t A|^2 + \delta(X,t) |\boldsymbol{B}|^2\right) \,dX\,dt
\end{equation}
is the density of a Carleson measure on $\Omega$ with Carleson norm $\|\mu\|_C$ and
\begin{equation}\label{E:carl4b}
\delta(X,t) |\nabla A|+\delta^2(X,t) |\partial_t A| + \delta(X,t) |\boldsymbol{B}| \le \|\mu\|_C^{1/2}.
\end{equation}
Then there exists $\varepsilon>0$ such that if for some $r_0>0$ $\max\{L^2,\|\mu\|_{C,r_0}\}<\varepsilon$ then the $L^p$ boundary value problem \eqref{E:para3}
is solvable for all $2\leq p < \infty$. Moreover, the estimate \eqref{E:estim3} holds.
\end{theorem}

The final theorem is a direct corollary of the following lemma we establish in this paper.

\begin{lemma}\label{L:Square2} Let $\Omega$ be an admissible domain from Definition \ref{D:domain} of character $(L,N,C_0)$. Let ${\mathcal L}=\partial_t-\di(A\nabla\cdot)-B\cdot\nabla$ be a parabolic operator with matrix $A$ satisfying uniform ellipticity with constants $\lambda$ and $\Lambda$,  \eqref{E:carl4a} be a density of a Carleson measure on $\Omega$ with Carleson norm $\|\mu\|_{C}$   and  \eqref{E:carl4b} holds. 
Then there exists a constant $C=C(\lambda,\Lambda,N,C_0)$  such that
for any solution $u$ with boundary data $f$ on any ball $\Delta_r\subset\partial\Omega$ with $r\le \min\{r_0/4, r_0/(4C_0)\}$ (c.f. Definition \ref{D:domain} for the meaning of $r_0$ and $C_0$) we have
\begin{equation}\label{E:sq10aaa}
\int_{T(\Delta_r)} |\nabla u |^2 x_0\,dX\,dt
 \leq     C(1+\|\mu\|_{C,2r}) (1+L^2)\int_{\Delta_{2r}} (N^{2r})^{2}(u) \,dX\,dt.
\end{equation}

Here $N^{2r}$ denotes the truncated non-tangential maximal function and $u$ is any locally bounded  solution
(that is $\|N^{2r}(u)\|_{L^\infty(\Delta_{2r})}<\infty$).
\end{lemma}

Using this result a new (significantly simplified) proof of $A_\infty$ property for parabolic operators (c.f. \cite{Rn, Rn2}) can be established. The paper \cite{Rn} states the result with conditions \eqref{E:carl4aa} and  \eqref{E:carl4bb} (and $B=0$)  and the paper \cite{Rn2} with the condition \eqref{E:carl4} (and $B=0$). \cite[Theorem 1.10]{HL} also contains a version of this result but with extra technical assumptions that were dealt with \cite{Rn}. However, \cite{Rn} does not allow first order terms, which \cite{HL} does handle. This realization allow us to state the next theorem for operators with first order terms satisfying large Carleson condition even though
 the doubling of parabolic measure is not known to be true in such case.

\begin{theorem}\label{T:Main3} Let $\Omega$ be an admissible domain from Definition \ref{D:domain}. Let ${\mathcal L}=\partial_t-\di(A\nabla\cdot)-B\cdot\nabla$ be a parabolic operator with matrix $A$ satisfying uniform ellipticity with constants $\lambda$ and $\Lambda$ and either (i) or (ii) holds, where:
\begin{itemize}
\item[(i)]  $\displaystyle d\mu=\left[\delta(X,t) |\nabla A|^2+ \delta^3(X,t) |\partial_t A|^2+ \delta(X,t)\sup_{B_{\delta(X,t)/2}(X,t)} \left| \boldsymbol{B} \right|^{2}\right] \,dX\,dt$
is the density of a Carleson measure on $\Omega$ with Carleson norm $\|\mu\|_C$ and
$$\delta(X,t) |\nabla A|+\delta^2(X,t) |\partial_t A|  \le \|\mu\|_C^{1/2}.$$
\item[(ii)] $d\mu=\left[\delta(X,t)^{-1} \sup_{1\le i,j\le n}\left( \osc_{B_{\delta(X,t)/2}(X,t)} a_{ij} \right)^2 + \delta(X,t)\sup_{B_{\delta(X,t)/2}(X,t)} \left| \boldsymbol{B} \right|^{2} \right] \,dX\,dt$ is the density of a Carleson measure on $\Omega$ with Carleson norm $\|\mu\|_C$.
\end{itemize}
Then there exists $p'>1$ such that the $L^p$ Dirichlet problem for the operator ${\mathcal L}$ on  $\Omega$ is solvable for all $p'<p\le\infty$.
\end{theorem}

This theorem provides no control over the size of $p'$ (apart from the trivial estimate from below by $1$). But it highlights {\it sharpness} of Theorems \ref{T:Main} and \ref{T:Main2} in the following sense. For every $1<p<\infty$ there exist operators ${\mathcal L}$ satisfying all assumptions of Theorem \ref{T:Main3} with large $\|\mu\|_{C}$ for which the $L^p$ Dirichlet problem is NOT solvable. Hence, the {\it smallness} condition in Theorems \ref{T:Main} and \ref{T:Main2} is necessary and CANNOT be removed.
\vspace{1mm}

\noindent{\it Proof of Theorem \ref{T:Main}.}  Remark 3 above Theorem  \ref{T:Main} provides reference that the parabolic measure under our assumptions is doubling.
The proof of the remaining statements uses the $L^2$ solvability of Lemma~\ref{L:Main}, perturbation argument using result from \cite{Sw} and interpolation.
For perturbation results of this type see also Chapter III of \cite{HL} and \cite{N}.
The main Lemma~\ref{L:Main} establishes $L^2$ solvability of the Dirichlet problem on domains with small Lipschitz constant when \eqref{E:carl6}
is the density of Carleson measure with small norm on all parabolic Carleson regions of size $\le r_0$.
To replace the condition (\ref{E:carl4}) by
(\ref{E:carl6}) we use the idea of \cite[Corollary 2.3]{DPP}. For a matrix $A$ satisfying
(\ref{E:carl4}) with boundeness and ellipticity constants $\lambda$ and $\Lambda$ one can find
(by mollifying the coefficients of $A$) a new matrix $\widetilde{A}$
with same boundedness and ellipticity constants such that the matrix $\widetilde{A}$
satisfies (\ref{E:carl6}) and
\begin{equation}
\sup\{\delta(X,t)^{-1}|(A-\widetilde{A})(Y,s)|^2;\, Y\in
B_{\delta(X,t)/2)}(X,t)\}dX\,dt\label{NPePert}
\end{equation}
is the density of a Carleson measure. Moreover, if the Carleson norm for matrix $A$
is small (on balls of radius $\le r_0$), so are the Carleson norms
of (\ref{E:carl6}) for $\widetilde{A}$ and (\ref{NPePert}). Hence Lemma~\ref{L:Main} gives us $L^2$ solvability of the Dirichlet problem
on $\Omega$ for the parabolic equation $v_t = \di (\widetilde{A} \nabla v)$.

To get $L^2$ solvability for our original equation $v_t = \di ({A} \nabla v)$ we apply \cite[Theorem 4]{Sw}. This theorem states that if ${\mathcal L}_0=\partial_t-\di (\widetilde{A} \nabla \cdot)$ and ${\mathcal L}_1=\partial_t-\di ({A} \nabla \cdot)$ are two parabolic operators whose difference satisfies (\ref{NPePert}) with sufficiently small Carleson measure, then the $L^2$ solvability for the operator ${\mathcal L}_0$ implies the same for the operator ${\mathcal L}_1$ (We are not using \cite[Theorem 4]{Sw}
in its full generality, but making choice $p=q=2$ with the measure $d\mu$ in the theorem being the measure $d\sigma$ from the Definition \ref{D:measure}). From this the $L^2$ solvability of a parabolic operator without a drift term $\boldsymbol{B}\cdot \nabla$ satisfying (\ref{E:carl4}) follows, provided the Carleson norm is sufficiently small. 

To include the drift term $\boldsymbol{B}\cdot \nabla$ it is necessary to revisit the proof given in \cite{Sw} in the light of results in \cite[Chapter III]{HL}.  The paper \cite{Sw} does not consider the drift term in the formulation of its main result but is forced to deal with it partially anyway (c.f. Lemma 2 for example where reflection across $Lip(1,1/2)$ boundary is mentioned). Further missing ingredients for adapting result of Sweezy to allow a small drift term (in terms of Carleson measure) are all in \cite{HL}, namely the issue of the parabolic measure being doubling if  a small drift term is present and the existence of a well-behaved Green's function in the presence of such drift term (c.f. Lemma 2.2 of Chapter III of this paper). With this in place the main result of Sweezy also holds for operators ${\mathcal L}_0=\partial_t-\di ({A} \nabla \cdot)$ and ${\mathcal L}_1=\partial_t-\di ({A} \nabla \cdot)+\boldsymbol{B}\cdot \nabla$ under the condition 
\begin{equation}
\sup\{\delta(X,t)|(\boldsymbol{B})(Y,s)|^2;\, Y\in
B_{\delta(X,t)/2)}(X,t)\}dX\,dt\nonumber
\end{equation}
has a small Carleson norm.

Finally, given the solvability of the continuous boundary value problem and the maximum principle $\|v\|_{L^\infty(\Omega)}\le \|f\|_{C_0(\partial\Omega)}$ the solvability for all values $2<p<\infty$ follows by interpolation.\qed
\vspace{1mm}

\noindent{\it Proof of Theorem \ref{T:Main2}.}  The Lemma \ref{L:Main} holds either with \eqref{E:carl6} or alternatively with
\eqref{E:carl4a} and \eqref{E:carl4b}. Either one of those yields (\ref{E:B1})-(\ref{E:A2}) for the parabolic equation on the flattened domain $U$. The rest of the argument
is identical to Theorem \ref{T:Main}.\qed
\vspace{1mm}

\noindent{\it Proof of Theorem \ref{T:Main3}.}  Consider first the case when $B=0$ and (i) holds. It follows from Lemma \ref{L:Square2} that if $u$ is a solution of a parabolic PDE with coefficients satisfying condition (i) with boundary data $f$, such that $\|f\|_\infty\le 1$
then for some $C=C(\lambda,\Lambda,L,N,C_0,\|\mu\|_{C})>0$ on all balls $\Delta\subset\partial\Omega$ we have
$$\frac1{\sigma(\Delta)}\int_{T(\Delta)} |\nabla u |^2 x_0\,dX\,dt
 \leq C.$$
 Here we have used the maximum principle implying $N^{2r}(u)\le 1$ and doubling property of measure $\sigma$. This type of condition readily implies $A_\infty$ of the measure associated with the operator $\mathcal L$ and in particular $L^p$ solvability for all (large) values of $p$.  In the elliptic case this has been established in \cite{KKPT2} and in the parabolic case (as well as elliptic with much simplified argument) in \cite{DPP2}. 
If condition (ii) holds then as in the proof of Theorem \ref{T:Main} a new matrix $\widetilde{A}$ can be constructed that satisfies (i). Hence ones gets $A_\infty$ property for the operator $\partial_t-\mbox{div}(\widetilde{A}\nabla\cdot)$ by the argument given above. To get the $L^p$ solvability of the original operator (with first order terms) we use the perturbation result \cite[Chapter III, Theorem 1.7]{HL} which does allow to handle first order terms with 
 $\delta(X,t)\sup_{B_{\delta(X,t)/2}(X,t)} \left| \boldsymbol{B} \right|^{2}dX\,dt$ having a large Carleson norm. As observed in \cite{HL} the $L^p$ solvability holds even though the doubling of the parabolic measure is not known for such operators. \qed

\section{Basic results and Interior estimates}\label{S13:BI}

In this section we state some basic results and interior estimates that will be needed later.
These two lemmas are modifed versions of Lemmas 2.3 and 2.4 in \cite{Di} for elliptic equations adapted to the parabolic setting.
\begin{lemma}\label{L:cover}
Let $E \subset \mathbb{R}^{n-1}\times\mathbb{R}$. Suppose that for each $(X, t) \in E$ a number $r(X,t)>0$ is given. Also assume that $\sup_{(X, t)\in E} r(X,t) < \infty$. Then there exists a sequence $(X_{i}, t_{i}) \in E$, $i=1,2,3,\dots$ such that all cubes $Q_{r_{i}} (X_i, t_i)$ ($r_i = r(X_i, t_i)$) are disjoint and
\begin{align*}
&(i) \  E \subset \bigcup_{i} Q_{3r_i} (X_i, t_i), \\
&(ii) \ \text{for all } (X, t) \in E, \text{ there exists $(X_j, t_j)$ such that } Q_{r(X,t)}(X, t) \subset Q_{5 r_j} (X_j, t_j) .
\end{align*}
\end{lemma}

\begin{lemma}\label{L:Equ} Let $r>0$ and $0< a < b$. Consider the non-tangential maximal functions defined using two set of cones cones $\Gamma_{a}^{r}$ and $\Gamma_{b}^{r}$.
Then for any $p>0$ there exists a constant $C_p>0$ such that
\[
N_{a}^{r} (u) \leq N_{b}^{r} (u),\quad \|N_{b}^{r}(u)\|_{L^p(\partial U)}\leq C_p\|N_{a}^{r}(u)\|_{L^p(\partial U)},
\]
for all $u:U\to\BBR$.
\end{lemma}

\begin{proof}
First of all, it is trivial to show
\[
N_{a}^{r} (u) \leq N_{b}^{r} (u),
\]
since the cone of smaller aperture $\Gamma_{a}^{r}$ is contained in $\Gamma_{b}^{r}$.

Our goal to show that, for any $\lambda > 0$, there exists a constant $C$ such that
\[
\left| \{ (x, t) \in \partial U : N_{b}^{r} (u) (x,t) > \lambda \} \right|
\leq C \left| \{ (x, t) \in \partial U : N_{a}^{r} (u) (x,t) > \lambda \} \right|.
\]

From this the claim $\|N_{b}^{r}(u)\|_{L^p(\partial U)}\leq C_p \|N_{a}^{r}(u)\|_{L^p(\partial U)}$ follows immediately, since for $\widetilde{E}(\lambda)=\{ (x, t) \in \partial U : N (u) (x,t) > \lambda \}$ we have
$$\int_{\partial U} N(u)(x,t)^p \,dX\,dt = c_p\int_{0}^\infty |\widetilde{E}(\lambda)|\lambda^{p-1}d\lambda,$$
and the estimate above gives us comparison of measures of the sets $\widetilde{E}(\lambda)$ for $N_a^r$ and $N_b^r$, respectively.

We make two simple geometrical observations. First, for any $(z_0,z,\tau) \in \Gamma_{b}^{r}(x,t)$ (that is $|z-x| + |t-\tau|^{1/2} < b z_0$), then $(x,t) \in Q_{b z_0} (z,\tau)$. Second, for $(y,s) \in Q_{ax_0/n}(x,t)$ and $0<x_0<r$ (that is, $|x_i - y_i| < a x_0/n$ for all $i$ and $|s-t|^{1/2} < a x_0/n$), then $(x_0, x, t) \in \Gamma_{a}^{r}(y,s)$.

Assume that
\[
(x, t) \in E(\lambda) = \{(y,s) \in \partial U : N_{b}^{r} (u) (y,s) > \lambda \}.
\]
It follows that, for some $(z_0, z, \tau) \in \Gamma_{b}^{r}(x,t)$, we have $|u(z_0, z, \tau)| > \lambda$. Therefore $(x,t) \in Q_{b z_0}(z,\tau)$ by the first observation. For any $(z', \tau') \in Q_{az_0/n}(z,\tau)$, the second observation implies that $(z_0, z, \tau) \in \Gamma_{a}^{r}(z',\tau')$. Hence $N_{a}^{r}(z', \tau') > \lambda$ and therefore
\[
Q_{az_0/n}(z,\tau) \subset E'(\lambda) = \{(y,s) \in \partial U : N_{a}^{r} (u) (y,s) > \lambda \}.
\]

Define $r(x,t)>0$ to be the smallest positive number such that $Q_{az_0/n}(z,\tau)\subset Q_{r(x,t)}(x,t)$. Due to the geometry of the nontangential cones
for some $K=K(a,b)>0$: $|Q_{r(x,t)}(x,t)|\le K |Q_{az_0/n}(z,\tau)|$.
If $\sup_{(x,t)\in E(\lambda)}r(x,t)=\infty$ there is nothing to prove as this implies that $E(\lambda)$ (and therefore also $\widetilde{E}(\lambda)$) contain balls of arbitrary large radius and hence both $|E(\lambda)|$ and $|\widetilde E(\lambda)|$  are infinite. So the claim holds. Otherwise we can apply Lemma~\ref{L:cover} and there exists a sequence of $\{(x_i, t_i)\} \subset E(\lambda)$ and $\{r_i\}$ such that
\[\begin{split}
|E(\lambda)| &\leq \sum_{i} \left| Q_{3 r_i} (x_i, t_i) \right| \\
             &\leq C \sum_{i} \left| Q_{r_i} (x_i, t_i) \right|\le CK \sum_i \left| Q_{a/nz_{0i}} (z_i, \tau_i) \right|\\
             &\leq CK |E'(\lambda)|,
\end{split}\]
the last inequality due to the fact that the sets $Q_{a/nz_{0i}} (z_i, \tau_i)$ are disjoint as $Q_{r_i} (x_i, t_i)$ are and are contained in $E'(\lambda)$.

For simplicity we have worked on the domain $U$; the upper half-space. However, a similar result holds on any admissible parabolic domain via the localization and the pull-back map $\rho$.
\end{proof}

Next, we state the two interior Cacciopoli estimates for the parabolic equations.
\begin{lemma}\label{L:Caccio}(A Cacciopoli inequality)
Suppose that $u$ is a weak solution\footnote{The weak solution is defined as usual; the equation is multiplied by a $C_0^\infty$ test function and integrated by parts in all variables moving all derivatives onto the test function. For details see  for example \cite[Chapter I, (2.10)-(2.11)]{HL}.} of \eqref{E:u}. For an interior point $(x_0, x, t) \in U$ (which means $x_0 > 0$) and any $0 <r < x_0 /4$ such that $Q_{4r}(X, t) \subset U$, there exists a constant $C$ such that
\begin{equation*}\begin{split}
& r^{n} \left(\sup_{Q_{r}(X, t)} u \right)^{2} \\
&\leq    C \sup_{t-(2r)^2 \leq s \leq t+(2r)^2} \int_{Q_{2r}(X,t)\cap{(\mathbb R}^n\times\{s\})} u^{2}(Y,s) \,dY
       + C\int_{Q_{2r}(X, t)} |\nabla u|^{2} \,dY\,ds \\
&\leq    \frac{C^2}{r^2} \int_{Q_{4r}(X, t)} u^{2}(Y, s) \,dY\,ds.
\end{split}\end{equation*}
\end{lemma}

The result is proven in \cite{HL} so we omit the proof. A similar claim holds for the second gradient if an additional assumption is placed on the coefficients.

\begin{lemma}\label{L:Caccio2}(A Cacciopoli inequality for the second gradient)
Suppose that $u$ is a weak solution of \eqref{E:u}. For an interior point $(x_0, x, t) \in U$ (which means $x_0 > 0$) and any $0 <r < x_0 /2$ such that $Q_{2r}(X, t) \subset U$,
assume that $|\nabla A|,|\boldsymbol{B}|\le K/r$ on $Q_{2r}(X, t)$. Then there exists a constant $C=C(K)$ such that
\begin{equation}\label{E:GC}\begin{split}
&\int_{Q_{r}(X, t)} |\nabla^2 u|^{2} \,dY\,ds\leq    \frac{C}{r^2} \int_{Q_{2r}(X, t)} |\nabla u|^{2} \,dY\,ds.
\end{split}\end{equation}
\end{lemma}

\begin{proof} Because $A$ is differentiable, without loss of generality we may assume that $u$ solves equation
of the form  \eqref{E:u} with matrix $A$ symmetric, i.e., $A=A^T$. Indeed we have
$$\partial_{x_i}(A_{ij}\partial_{x_j} u)=\partial_{x_j}(A_{ij}\partial_{x_i}u) +(\partial_{x_i} A)\partial_{x_j}u-(\partial_{x_j}A)\partial_{x_i} u$$
and hence the matrix $A$ can be symmetrized at the expense of a first order (drift) term.

We take the spatial gradient of the PDE \eqref{E:u}. For simplicity, let $v_i = \partial_{x_i} u$ and $w_i=v_i\zeta^2$, $0\le i\le n-1$ where $0\leq \zeta \leq 1$ is a smooth cutoff function equal to $1$ on $Q_{r}(X,t)$ and supported in $Q_{2r}(X,t)$ satisfying $r|\nabla \zeta| + r^2 |\zeta_{t}| \leq c$ for some $c>0$. It follows that (summing over repeating indices)
\[
\int_{Q_{2r}} (v_{i})_t w_i \,dX\,dt = -\int_{Q_{2r}} \left( A \nabla v_i + (\partial_{x_i} A) \nabla u \right)\cdot\nabla w_i +\boldsymbol{B}\cdot(w_i\nabla v_i-(\partial_{x_i}w_i)\nabla u)\,dX\,dt,
\]
which implies that (due to the symmetry of $A$ some terms do not appear below):
\[\begin{split}
&\frac{1}{2}\int_{Q_{2r}} \left[(|\nabla u|\zeta)^2\right]_{t} \,dX\,dt
+ \int_{Q_{2r}} A \nabla(v_i\zeta)\cdot \nabla(v_i\zeta)\,dX\,dt \\
&=  \int_{Q_{2r}} |\nabla u|^{2}\zeta\zeta_{t} \,dX\,dt  + \int_{Q_{2r}} |\nabla u|^2 A\nabla\zeta\cdot\nabla \zeta \,dX\,dt \\
&\quad - \int_{Q_{2r}}  (\partial_{x_i} A)(\nabla u)\zeta\cdot \nabla (v_i\zeta)\,dX\,dt  - \int_{Q_{2r}} (\partial_{x_i} A)v_i\nabla u\cdot \zeta\nabla\zeta \,dX\,dt\\
&\quad +  \int_{Q_{2r}} \boldsymbol B\cdot (v_i\zeta)\nabla(v_i\zeta)\,dX\,dt -  \int_{Q_{2r}}|\nabla u|^2\boldsymbol B\cdot \zeta \nabla \zeta\,dX\,dt\\
&\quad -  \int_{Q_{2r}} \boldsymbol B\cdot(\nabla u)\zeta\partial_{x_i}(v_i\zeta))\,dX\,dt +\int_{Q_{2r}} \boldsymbol B\cdot v_i\nabla u (\zeta\partial_{x_i}\zeta)\,dX\,dt.
\end{split}\]

Using the ellipticity and boundedness of the coefficients and the Cauchy-Schwarz inequality, it follows that
\[\begin{split}
& \sup_{t-(2r)^2 \leq s \leq t+{2r}^2} \int_{Q_{r}(X,t)\cap{(\mathbb R}^n\times\{s\})} |\nabla u|^{2}(X,s) \,dX
+ \lambda \int_{Q_{r}} \left|\nabla^2 u\right|^{2}\,dX\,dt \\
&\leq \frac{2c}{r^2} (1+\Lambda) \int_{Q_{2r}} |\nabla u|^{2} \,dX\,dt
 + \frac{C'}{\lambda} \int_{Q_{2r}} \left(|\nabla A|^{2} + |\boldsymbol{B}|^{2}\right) |\nabla u|^{2}\,dX\,dt \\
& + \frac{C'}{\lambda r} \int_{Q_{2r}} \left(|\nabla A| + |\boldsymbol{B}|\right) |\nabla u|^2 \,dX\,dt \\
&\leq \frac{C}{r^2} \int_{Q_{2r}} |\nabla u|^2 \,dX\,dt
\end{split}\]
for some constant $C= C(\lambda, \Lambda, c, K)$. Then \eqref{E:GC} follows by dropping the first term on the left hand side.
\end{proof}

We will need the Poincar\'e inequality for functions vanishing at the boundary:

\begin{lemma}\label{L:PIg} Let $\Omega\subset {\BBR}^n$ be a bounded domain. There exists $c_n>0$ depending only on the dimension such that if
$u \in W^{1,2}_{0}(\Omega)$ and diam $(\Omega)=\sup_{x,y\in\Omega}|x-y|= R$  then
\[
\int_{\Omega} u^2 \,dX \leq c_nR^2 \int_{\Omega} |\nabla u|^{2} \,dX .
\]
\end{lemma}
To see the lemma above recall that the first Dirichlet eigenvalue of Laplacian on domain $\Omega$ is the infimum
$$\lambda_{\Omega}=\inf_{u\in W^{1,2}_0(\Omega)}\frac{\|\nabla u\|^2_2}{\|u\|^2_2}.$$
It immediately follows (by extension) that if $\Omega\subset\Omega'$ then $\lambda_{\Omega}\ge\lambda_{\Omega'}$. The optimal constant in the Poincar\'e inequality is $\lambda_{\Omega}^{-1}$. Hence if
$\Omega\subset B_R$ (a ball of radius $R$) then 
$$\lambda_{\Omega}^{-1}\le \lambda^{-1}_{B_R}= c_nR^2,$$
from which the result follows.

Lemmas 3.4 and 3.5 in \cite{HL} give us the following estimates for a weak solution of \eqref{E:u}.

\begin{lemma}\label{Holder}(Interior H\"{o}lder continuity)
Suppose that $u$ is a weak solution of \eqref{E:u} in U. If $|u| \leq K < \infty$ for some constant $K>0$ in $Q_{4r}(x_0, x, t) \subset U$, then for any $(y_0, y, s), (z_0, z, \tau) \in Q_{2r}(x_0, x, t)$ there exists a constant $C>0$ and $0 < \alpha < 1$ such that
\[
\left|u(y_0, y, s) - u(z_0, z, \tau)\right| \leq C K \left( \frac{|y_0 - z_0| + |y - z| + |s - \tau|^{1/2}}{r} \right)^{\alpha}.
\]
\end{lemma}

\begin{lemma}\label{Harnack}(Harnack inequality)
Suppose that $u$ is a weak nonnegative solution of \eqref{E:u} in $U$ such that $Q_{4r}(X, t) \subset U$. Suppose that $(Y,s), (Z, \tau) \in Q_{2r}(X,t)$. There exists an a priori constant $c$ such that, for $\tau < s$,
\[
u(Z, \tau) \leq u(Y, s) \exp \left[ c\left( \frac{|Y-Z|^2}{|s-\tau|} + 1 \right) \right].
\]
If $u\geq 0$ is a weak solution of the adjoint operator of \eqref{E:u}, then this inequality is valid when $\tau > s$.
\end{lemma}

We state a version of the maximum principle, that is a modification of Lemma 3.38 from \cite{HL}.
\begin{lemma}\label{L:MP}(Maximum Principle) Let $u$, $v$ be bounded continuous local weak solutions to \eqref{E:u} in $\Omega$ where $\Omega$ is an admissible parabolic domain and $A$ and $B$ satisfy \eqref{E:UEB}, \eqref{E:B1}, and \eqref{E:A1}. If $|u|,|v|\to 0$ uniformly as $t\to-\infty$ and
\[
\limsup_{(Y,s)\to (X,t)} (u-v)(Y,s) \leq 0
\]
for all  $(X,t)\in\partial\Omega$, then $u\leq v$ in $\Omega$.
\end{lemma}
\begin{proof} The argument is essentially the same as in Lemma 3.38 from \cite{HL}. Due to continuity of the solutions and the assumption that $|u|,|v|\to 0$ uniformly as $t\to-\infty$ for any $\epsilon>0$ and $T<\infty$ there exists a compact set $K$ such that $u-v\le \epsilon$ for all $(X,t)\in\Omega\setminus K$
with $t\le T$. On $K$ coefficients $A$, $B$ are essentially bounded by \eqref{E:B1} and \eqref{E:A1} hence the weak maximum principle holds on $K$. Using it we obtain $u-v\le \epsilon$ on $K$. It follow that $(u-v)(X,t)\le \epsilon$ for all $(X,t)\in\Omega$ such that $t\le T$. As $T$ can be chosen arbitrary large, it follows
that $(u-v)\le \epsilon$ on $\Omega$. Hence the claim holds.
\end{proof}

\noindent{\it Remark.} The lemma is also applicable in case when $u\le v$ on the boundary of $\Omega\cap\{t\ge \tau\}$ for a given time $\tau$. Obviously then the assumption $|u|,|v|\to 0$ uniformly as $t\to-\infty$ is not necessary. Another important case, when Lemma as stated here applies is when $u|_{\partial\Omega},v|_{\partial\Omega}\in C_0(\partial\Omega)$ where $C_0(\partial\Omega)$ denotes the class of continuous functions decaying to zero as $t\to\pm\infty$. This class is dense in any $L^p(\partial\Omega,d\sigma)$, $p<\infty$ allowing us to consider an extension of the solution operator from $C_0(\partial\Omega)$ to $L^p(\partial\Omega,d\sigma)$.

\section{An estimate of the square function of a solution}\label{S31:Square}

In this section we find an $L^2$ estimate of the square function of a solution by the boundary data and the non-tangential maximal function.

\begin{lemma}\label{L:Square} Let $\Omega$ be a domain satisfying Definition \ref{D:domain} with smooth boundary $\partial\Omega$.
Let $u$ be any weak solution of \eqref{E:u} satisfying \eqref{E:UEB}, \eqref{E:B1}, \eqref{E:B2}, \eqref{E:A1}, and \eqref{E:A2} with Dirichlet boundary data $f \in L^{2}(\partial \Omega)$. Then there exist positive constants $C_1$ and $C_2$ independent of $u$ such that  $r_0>0$ small we have
\begin{equation}\label{E:sq10aa}\begin{split}
&\frac{C_1}2\int_{0}^{r_0/2}\int_{\partial\Omega} |\nabla u |^2 x_0\,dx\,dt\,dx_0
+ \frac{2}{r_0} \int_{0}^{r_0}\int_{\partial\Omega} u^{2}(x_0,x,t)\,dx\,dt\,dx_0 \\
&\leq    \int_{\partial\Omega} u^2(r_0,x,t)\,dx\,dt
       + \int_{\partial\Omega} u^2(0,x,t) \,dx\,dt\nonumber\\
       &+ C_2(\|\mu_1\|_{C,2r_0} +\|\mu_2\|_{C,2r_0}+\|\mu_2\|^{1/2}_{C,2r_0}) \int_{\partial\Omega} \left[N^{r_0}(u)\right]^2 \,dx\,dt.
\end{split}\end{equation}
\end{lemma}

\noindent {\it Proof of Lemmas \ref{L:Square2} and \ref{L:Square}}. 
Both lemmas are proven at the same time. We start with Lemma \ref{L:Square} and then address the local result (Lemma \ref{L:Square2}). We begin with a local estimate on a parabolic ball $Q_r(y,s)$, for a point $(y, s)\in \partial U$ and a radius $r>0$ to be determined later, by considering the expression
\begin{equation}\label{Sq00}
2\sum_{i,j=0}^{n-1}\int_{0}^{r}\int_{Q_{2r}(y,s)} \frac{a_{ij}}{a_{00}} u_{x_i} u_{x_j} x_0 \zeta^2 \,dx\,dt\,dx_0
\end{equation}
where $\zeta$ is a cutoff function independent of the $x_0 -$variable satisfying
\[
\zeta = \begin{cases}
        1 & \text{in} \quad Q_{r}(y,s), \\
        0 & \text{outside} \quad Q_{2r}(y,s),
        \end{cases}
\]
such that for some constant $0< c < \infty$
\[
r |\partial_{x_{i}} \zeta| + r^2 |\zeta_{t}| \leq c \quad \text{where} \quad 1\leq i \leq n-1 .
\]
For brevity, let $Q_{r}=Q_{r}(y,s)$ and $Q_{2r}=Q_{2r}(y,s)$. Because of the cutoff function $\zeta$ and the uniform ellipticity and boundedness of the matrix $A$, the quantity \eqref{Sq00} is bounded below by
\begin{equation}\label{Sq01}
\frac{2\lambda}{\Lambda} \int_{0}^{r}\int_{Q_{r}} \left| \nabla u \right|^{2} x_0 \,dx\,dt\,dx_0\le 2\sum_{i,j}\int_{0}^{r}\int_{Q_{2r}} \frac{a_{ij}}{a_{00}} u_{x_i} u_{x_j} x_0 \zeta^2 \,dx\,dt\,dx_0,
\end{equation}
where the expression on the left-hand side of (\ref{Sq01}) represent a piece of the $L^2$ norm of the square function truncated to the Carleson region  $T(Q_r)$.

To estimate the right-hand side of \eqref{Sq00}, we integrate by parts in the $x_i -$variable (note that the outer normal vector is $\nu=(1,0,\ldots,0)$ because the domain $U$ is just $\{x_0>0\}$). From now on we use the 
Einstein notation and sum over repeating indices. Recall the both $i$ and $j$ are summed from $0$ to $n-1$.
We get
\begin{equation}\begin{split}\label{Sq0}
&2\int_{0}^{r}\int_{Q_{2r}} \frac{a_{ij}}{a_{00}} u_{x_i} u_{x_j} x_0 \zeta^2 \,dx\,dt\,dx_0
=   2\int_{Q_{2r}} \frac{a_{0j}}{a_{00}} u(r,x,t) u_{x_j}(r,x,t) r \zeta^2 \,dx\,dt \\
&\quad   - 2\int_{0}^{r}\int_{Q_{2r}} \frac{1}{a_{00}} u \partial_{x_{i}}\left(a_{ij}u_{x_j}\right) x_0 \zeta^2 \,dx\,dt\,dx_0 \\
&\quad - 2\int_{0}^{r}\int_{Q_{2r}} \partial_{x_{i}}\left(\frac{a_{ij}}{a_{00}}\right) u u_{x_j} x_0 \zeta^2 \,dx\,dt\,dx_0 \\
&\quad   - 4\int_{0}^{r}\int_{Q_{2r}} \frac{a_{ij}}{a_{00}} u u_{x_j} x_0 \zeta \zeta_{x_i} \,dx\,dt\,dx_0
         - 2\int_{0}^{r}\int_{Q_{2r}} \frac{a_{0j}}{a_{00}} u u_{x_j} \zeta^2 \,dx\,dt\,dx_0 \\
& = I + II + III + IV + V .
\end{split}\end{equation}

We use the parabolic PDE \eqref{E:u} to split the second term $II$ into two new terms
\[\begin{split}
 II
& = - 2\int_{0}^{r}\int_{Q_{2r}} \frac{1}{a_{00}} u u_{t} x_0 \zeta^2 \,dx\,dt\,dx_0\\
    &+ 2\sum_{i}\int_{0}^{r}\int_{Q_{2r}} \frac{1}{a_{00}} b_i u u_{x_i} x_0 \zeta^2 \,dx\,dt\,dx_0  = II_{1} + II_{2}.
\end{split}\]

We take the integration by parts with respect to $x_0$-variable by observing that $2 x_0 = \partial_{x_0} x_0^2$. This gives
\[\begin{split}
II_{1}
&= - \int_{0}^{r}\int_{Q_{2r}} \frac{1}{a_{00}} u u_{t} \left(\partial_{x_0} x_0^2\right) \zeta^2 \,dx\,dt\,dx_0 \\
&= - \int_{Q_{2r}} \frac{1}{a_{00}} u(r,x,t) u_{t}(r,x,t) r^2 \zeta^2 \,dx\,dt\\ &+ \int_{0}^{r}\int_{Q_{2r}} \partial_{x_0}\left(\frac{1}{a_{00}}\right) u u_{t} x_0^2 \zeta^2 \,dx\,dt\,dx_0 \\
&\quad + \int_{0}^{r}\int_{Q_{2r}} \frac{1}{a_{00}} u_{x_0} u_{t} x_0^2 \zeta^2 \,dx\,dt\,dx_0
       + \int_{0}^{r}\int_{Q_{2r}} \frac{1}{a_{00}} u \left(\partial_{x_0}u_{t}\right) x_0^2 \zeta^2 \,dx\,dt\,dx_0 \\
&= II_{11} + II_{12} + II_{13} + II_{14}.
\end{split}\]

First, we analyze $II_{11}$ by integrating by parts in the $t$-variable
\[\begin{split}
II_{11}
&= - \frac{1}{2}\int_{Q_{2r}} \frac{1}{a_{00}} \partial_{t}\left(u^2\right)(r,x,t) r^2 \zeta^2 \,dx\,dt \\
&= \frac{1}{2}\int_{Q_{2r}} \partial_{t}\left(\frac{1}{a_{00}}\right) u^2(r,x,t) r^2 \zeta^2 \,dx\,dt
 + \int_{Q_{2r}} \frac{1}{a_{00}} u^2(r,x,t) r^2 \zeta\zeta_{t} \,dx\,dt \\
&= II_{111} + II_{112}.
\end{split}\]
Hence the first term of this expression is bounded by
\[
II_{111} \leq \frac{1}{2\lambda^2}\int_{Q_{2r}} |A_{t}| u^2(r,x,t) r^2 \zeta^2 \,dx\,dt.
\]
Next, we bound the term $II_{12}$ using the area function we have defined previously.
\[\begin{split}
&II_{12}
= - \int_{0}^{r}\int_{Q_{2r}} \frac{\partial_{x_0} a_{00}}{a_{00}^{2}} u u_{t} x_0^2 \zeta^2 \,dx\,dt\,dx_0 \\
&\leq \frac{1}{\lambda^2}\left(\int_{0}^{r}\int_{Q_{2r}} x_0 \left|\nabla A\right|^2 u^2 \zeta^2 \,dx\,dt\,dx_0\right)^{1/2}
      \left(\int_{0}^{r}\int_{Q_{2r}} \left|u_{t}\right|^2 x_0^3 \zeta^2 \,dx\,dt\,dx_0\right)^{1/2}.
\end{split}\]
In the term $II_{14}$, we switch the order of derivatives ($\partial_{x_0}u_t=\partial_{t} u_{x_0}$) and then carry out integration by parts in the $t-$variable.
\[\begin{split}
II_{14}
&=     - \int_{0}^{r}\int_{Q_{2r}} \partial_{t}\left(\frac{1}{a_{00}}\right) u u_{x_0} x_0^2 \zeta^2 \,dx\,dt\,dx_0 \\
&\quad - \int_{0}^{r}\int_{Q_{2r}} \frac{1}{a_{00}} u_{t} u_{x_0} x_0^2 \zeta^2 \,dx\,dt\,dx_0- 2\int_{0}^{r}\int_{Q_{2r}} \frac{1}{a_{00}} u u_{x_0} x_0^2 \zeta\zeta_{t} \,dx\,dt\,dx_0 \\
& = II_{141} + II_{142} + II_{143}.
\end{split}\]
We observe that
\[\begin{split}
&II_{141}
= \int_{0}^{r}\int_{Q_{2r}}\frac{\partial_{t} a_{00}}{a_{00}^{2}} u u_{x_0} x_0^2 \zeta^2 \,dx\,dt\,dx_0 \\
&\leq \frac{1}{\lambda^2}\left(\int_{0}^{r}\int_{Q_{2r}} x_0^3 \left|A_{t}\right|^2 u^2 \zeta^2 \,dx\,dt\,dx_0\right)^{1/2}
      \left(\int_{0}^{r}\int_{Q_{2r}} \left|\nabla u\right|^2 x_0 \zeta^2 \,dx\,dt\,dx_0\right)^{1/2},
\end{split}\]
and
\[
II_{142} = - II_{13} .
\]

By the Cauchy-Schwarz inequality we have for $II_2$:
\[
II_{2}
\leq \frac{2n}{\lambda} \left(\int_{0}^{r}\int_{Q_{2r}} x_0 |\boldsymbol{B}|^2 u^2 \zeta^2 \,dx\,dt\,dx_0\right)^{1/2}
       \left(\int_{0}^{r}\int_{Q_{2r}}|\nabla u|^2 x_0 \zeta^2 \,dx\,dt\,dx_0 \right)^{1/2}.
\]

Next, we look at $III$:
\[\begin{split}
& III
=   2\sum_{i,j} \int_{0}^{r}\int_{Q_{2r}} \frac{a_{ij}\partial_{x_i} a_{00}-a_{00}\partial_{x_i}a_{ij}}{a_{00}^{2}} u u_{x_j} x_0 \zeta^2 \,dx\,dt\,dx_0 \\
&\leq \frac{4 n^2\Lambda}{\lambda^{2}} \left(\int_{0}^{r}\int_{Q_{2r}} x_0 |\nabla A|^2 u^2 \zeta^{2} \,dx\,dt\,dx_0\right)^{1/2}
      \left(\int_{0}^{r}\int_{Q_{2r}} |\nabla u|^2 \zeta^{2} x_0 \,dx\,dt\,dx_0 \right)^{1/2}.
\end{split}\]

The last term we look at is the integral quantity $V$ by considering two cases $j=0$ and $j\neq 0$. First for $j=0$, we have
\[\begin{split}
V_{\{j=0\}}
&= -\int_{0}^{r}\int_{Q_{2r}} \partial_{x_0}(u^{2}) \zeta^{2} \,dx\,dt\,dx_0 \\
&= -\int_{Q_{2r}} u^2(r,x,t) \zeta^2 \,dx\,dt + \int_{Q_{2r}} u^2(0,x,t) \zeta^2 \,dx\,dt
\end{split}\]
When $j\ne 0$ integrating by parts further using  $1=\partial_{x_0} x_0$ we get
\[\begin{split}
V_{\{j \neq 0\}}
&= -2\int_{0}^{r}\int_{Q_{2r}} \frac{a_{0j}}{a_{00}} u u_{x_j}  \left(\partial_{x_0} x_0\right)  \zeta^2 \,dx\,dt\,dx_0 \\
&= -2\int_{Q_{2r}} \frac{a_{0j}}{a_{00}} u(r,x,t) u_{x_j}(r,x,t) r \zeta^2 \,dx\,dt\\
   &+ 2\int_{0}^{r}\int_{Q_{2r}} \partial_{x_0} \left(\frac{a_{0j}}{a_{00}}\right) u u_{x_j} x_0 \zeta^2 \,dx\,dt\,dx_0 \\
&\quad + 2\int_{0}^{r}\int_{Q_{2r}} \frac{a_{0j}}{a_{00}} u_{x_0} u_{x_j} x_0 \zeta^2 \,dx\,dt\,dx_0\\
       &+ 2\int_{0}^{r}\int_{Q_{2r}} \frac{a_{0j}}{a_{00}} u \left(\partial_{x_0 x_j} u\right) x_0 \zeta^2 \,dx\,dt\,dx_0 \\
&= V_{1} + V_{2} + V_{3} + V_{4}.
\end{split}\]
Observe that
\[
V_{1} = -I_{\{j \neq 0\}}.
\]
It follows that
\[
V_{2}
= 2\int_{0}^{r}\int_{Q_{2r}}\frac{a_{00}\partial_{x_0}a_{0j} - a_{0j}\partial_{x_0}a_{00}}{a_{00}^{2}} u u_{x_j} x_0 \zeta^2 \,dx\,dt\,dx_0
\]
and therefore
\[\begin{split}
&\left|\sum_{j\neq 0} V_{2}\right| \leq \\&\frac{4 n \Lambda}{\lambda^{2}}
                           \left(\int_{0}^{r} \int_{Q_{2r}} x_0 |\nabla A|^{2} u^{2}  \zeta^{2} \,dx\,dt\,dx_0\right)^{1/2}
                           \left(\int_{0}^{r} \int_{Q_{2r}}|\nabla u|^{2} x_0 \zeta^2 \,dx\,dt\,dx_0\right)^{1/2}.
\end{split}
\]
To study $V_{4}$, we take advantage that $j\neq 0$. We switch the order of derivatives so that we work with $\partial_{x_j x_0} u$  and take the integration by parts in the $x_j$-variable. This give us
\[\begin{split}
V_{4}
&= -2\int_{0}^{r} \int_{Q_{2r}} \partial_{x_j}\left(\frac{a_{0j}}{a_{00}}\right) u u_{x_0} x_0 \zeta^2 \,dx\,dt\,dx_0 \\
&\quad - 2\int_{0}^{r} \int_{Q_{2r}} \frac{a_{0j}}{a_{00}} u_{x_j} u_{x_0} x_0 \zeta^2 \,dx\,dt\,dx_0
       - 4\int_{0}^{r} \int_{Q_{2r}} \frac{a_{0j}}{a_{00}} u u_{x_0}  x_0 \zeta \zeta_{x_j} \,dx\,dt\,dx_0 \\
&= V_{41} + V_{42} + V_{43}.
\end{split}\]
As with $V_2$, we have the same upper bound for $V_{41}$
\[\begin{split}
&\left|\sum_{j\neq 0}V_{41}\right|
\leq\\ &\frac{4 n \Lambda}{\lambda^2} \left(\int_{0}^{r}\int_{Q_{2r}} x_0 |\nabla A|^2 u^2 \zeta^2 \,dx\,dt\,dx_0\right)^{1/2}
                                    \left(\int_{0}^{r}\int_{Q_{2r}} |\nabla u|^{2} x_0 \zeta^2 \,dx\,dt\,dx_0 \right)^{1/2}.
\end{split}
\]
Next,
\[
V_{42} = - V_{3} .
\]

We now put together all terms we have encountered (and that did not cancel out). There are 4 types of terms:
\begin{align*}
J_1 &= I_{\{j=0\}} + II_{111} + V_{\{j=0\}}, \\
J_2 &= II_{12} \\
J_3 &= II_{141} + II_{2} +III + \sum_{j\neq 0} V_{2} + \sum_{j\neq 0} V_{41}\\
J_4 &= II_{112} + II_{143} + IV + \sum_{j \neq 0} V_{43}.
\end{align*}

The following crucial result will be used for terms containing $\nabla A$ or $\boldsymbol{B}$. For any function $u$ and a Carleson measure $\mu$ we have that
\[
\int_{U} |u|^{2} \,d\mu \leq \|\mu\|_C \|N(u)\|^{2}_{L^{2}(\mathbb{R}^{n})},
\]
with a local version of this statement (on any Carleson box) holding as well.

The first term we use this result for is $J_2$. Since $\mu_2$ is a Carleson measure we have:
\[
J_2
\leq \frac{1}{\lambda^2} \left(\|\mu_2\|_{C,2r} \int_{Q_{2r}} \left[N^{r}(u)\right]^2 \,dx\,dt\right)^{1/2}
          \left(\int_{0}^{r}\int_{Q_{2r}} |u_t|^2 x_0^3 \zeta^2 \,dx\,dt\,dx_0\right)^{1/2} .
\]
With a constant
\[
C_1 = \max \left\{ \frac{4n^2\Lambda + 8n\Lambda}{\lambda^2}, \frac{2n}{\lambda}, \frac{1}{\lambda^2} \right\},
\]
it follows using \eqref{E:B1}-\eqref{E:A2} that
\[\begin{split}
J_{3}
&\leq C_1 \left(\int_{0}^{r}\int_{Q_{2r}} \left( x_0 |\nabla A|^2 + x_0 |\boldsymbol{B}|^2 + x_0^3 |A_t|^2 \right) u^2 \zeta^2 \,dx\,dt\,dx_0\right)^{1/2} \\
&\quad \times \left(\int_{0}^{r}\int_{Q_{2r}} |\nabla u|^2 x_0 \zeta^2 \,dx\,dt\,dx_0\right)^{1/2} \\
&\leq C_1 \left( (\|\mu_1\|_{C,2r} +\|\mu_2\|_{C,2r}) \int_{Q_{2r}} N_{r}^{2}(u) \,dx\,dt\right)^{1/2}\\
&\quad\times          \left(\int_{0}^{r}\int_{Q_{2r}} |\nabla u|^2 x_0 \zeta^2 \,dx\,dt\,dx_0\right)^{1/2}.
\end{split}\]

Moreover, due to \eqref{E:A1} we have
\[
\frac{1}{2\lambda^2}\int_{Q_{2r}} r^2 |A_{t}| u^2(r,x,t)\zeta^2 \,dx\,dt
 \leq \frac{\|\mu_2\|^{1/2}_{C,2r}}{2\lambda^2} \int_{Q_{2r}}\left[N^{r}(u)\right]^2 \,dx\,dt.
\]

Hence, it follows that
\begin{equation}\label{Sq08}\begin{split}
& 2\int_{0}^{r}\int_{Q_{2r}} \frac{a_{ij}}{a_{00}} u_{x_i} u_{x_j} x_0 \zeta^2 \,dx\,dt\,dx_0 =J_1 + J_2 + J_3 + J_4 \\
&\leq    \int_{Q_{2r}} \partial_{x_0} [u^2(r,x,t)] r \zeta^2 \,dx\,dt
       + \frac{\|\mu_2\|^{1/2}_{C,2r}}{2\lambda^2} \int_{Q_{2r}}\left[N^{r}(u)\right]^2 \,dx\,dt \\
&\quad - \int_{Q_{2r}} u^2(r,x,t) \zeta^2 \,dx\,dt
       + \int_{Q_{2r}} u^2(0,x,t) \zeta^2 \,dx\,dt \\
&\quad + \frac{1}{\lambda^2} \left( \|\mu_2\|_{C,2r} \int_{Q_{2r}} \left[N^{r}(u)\right]^2 \,dx\,dt\right)^{1/2}
          \left(\int_{0}^{r}\int_{Q_{2r}} |u_t|^2 x_0^3 \zeta^2 \,dx\,dt\,dx_0\right)^{1/2}\\
&\quad +  C_1 \left( (\|\mu_1\|_{C,2r} +\|\mu_2\|_{C,2r}) \int_{Q_{2r}} \left[N^{r}(u)\right]^2 \,dx\,dt\right)^{1/2}\\
&\quad\times          \left(\int_{0}^{r}\int_{Q_{2r}} |\nabla u|^2 x_0 \zeta^2 \,dx\,dt\,dx_0\right)^{1/2}\\
&\quad + J_4 .
\end{split}\end{equation}

We now use (\ref{Sq08}) to obtain a global estimate on a collar neighborhood of $\Omega$. Recall, that in addition to Definition \ref{D:domain} we have also assumed that $\partial\Omega$ is smooth. It follows that there exists a collar neighborhood $V$ of $\partial\Omega$ in ${\BBR}^{n+1}$ such that $\Omega\cap V$ can be
parameterized as $(0,r)\times\partial\Omega$ for some small $r>0$. These new coordinates are defined as follows.

Consider a smooth function $G:V\to {\mathbb S}^{n+1}$ such that for each $(Y,s)\in V$
the unit vector $G(Y,s)$ is in \lq good' direction (see subsection \ref{S:alld}). Given a boundary point $(X,\tau)\in\partial \Omega$ we solve the ODE
$$\gamma'(s)= G(\gamma(s)),\qquad \gamma(0)=(X,\tau)$$
and set $(x_0,X,\tau)=\gamma(x_0)$ for all $x_0>0$ small so that $\gamma(x_0)\in V\cap \Omega$.

We also introduce local coordinates on $\partial\Omega$ to parameterize $(X,\tau)\in\partial\Omega$. We consider local coordinate chart $\varphi$ from a neighborhood $Q_{2r}(0,0)$ of a point $(0,0)\in \partial U$
to a neighborhood of a point in $\partial\Omega$. Then the map
$$(x_0,x,t)\mapsto (x_0,\varphi(x,t))$$
maps neighborhood of $(0,0,0)$ in $\overline{U}$ to a neighborhood in $\overline{V\cap\Omega}$ of a point in $\partial\Omega$.

We choose $r>0$ small enough so that for all $0<x_0\le 2r$ and $(0,x,t)\in\partial U$ the point $(x_0,\varphi(x,t))\in V\cap \Omega$. It follows from the Definition \ref{D:domain} that there is a collection of coordinate charts covering $\partial\Omega$, with each point belonging to at most $K=K(N,n)<\infty$ different charts. Consider a partition of unity subordinate to this collection, and let $\{\zeta_{k}\}_{k=1}^{\infty}$, such that for all $k$
\[
\zeta_k = \begin{cases}
        1 & \text{in} \quad Q_{r}(y_k,s_k), \\
        0 & \text{outside} \quad Q_{2r}(y_k,s_k),
        \end{cases}
\]
for some constant $0< c=c(n) < \infty$
\[
r |\partial_{x_{i}} \zeta_k| + r^2 |\partial_{t} \zeta_{k}| \leq c \qquad \quad 1\leq i \leq n-1
\]
and $\sum_{k} \zeta^2_{k} = 1$ everywhere. Now we take the sum of expressions
$$2\int_{0}^{r}\int_{Q_{2r}} \frac{a_{ij}}{a_{00}} u_{x_i} u_{x_j} x_0 \zeta^2 \,dx\,dt\,dx_0$$
over all coordinate charts. Note that this expression is independent of the choice of coordinate map $\varphi$, as $x_0$ and $a_{00}$ do not depends on $\varphi$
(the variable $x_0$ is global). Hence, using (\ref{Sq08}) we obtain a lower bound for
$$\frac2{\Lambda}\int_{0}^{r}\int_{\partial\Omega} (A\nabla u\cdot \nabla u) x_0  \,dx\,dt\,dx_0$$
which is an expression comparable to $\|S^r(u)\|^2_{L^2(\partial\Omega)}$ (this is the truncated square function at height $r$).

The terms $J_4$ in (\ref{Sq08}) all contain terms of the type $\zeta_k\zeta_{k\,x_i}$ or $\zeta_k\zeta_{k\,t}$ which add up to zero when summed
over all partitions (since $\sum_{k} \zeta^2_{k} = 1$). This yields
\begin{equation}\label{Sq09}\begin{split}
&\frac{2\lambda}{\Lambda}\|S^r(u)\|^2_{L^2(\partial\Omega)}=\frac{2\lambda}{ \Lambda}\int_{0}^{r}\int_{\partial\Omega} |\nabla u |^2  x_0 \,dx\,dt\,dx_0 \\
&\leq    \int_{\partial\Omega} \left(\partial_{x_0} u^2\right)(r,x,t) r  \,dx\,dt
       + \frac{K\|\mu_2\|^{1/2}_{C,2r}}{2\lambda^2} \int_{\partial\Omega}\left[N^{r}(u)\right]^2 \,dx\,dt \\
&\quad - \int_{\partial\Omega} u^2(r,x,t) \,dx\,dt
       + \int_{\partial\Omega} u^2(0,x,t) \,dx\,dt \\
&\quad + \frac{\|\mu_2\|_{C,2r}}{2\lambda^2} K(\eta)\int_{\partial\Omega} \left[N^{r}(u)\right]^2 \,dx\,dt+
          \eta\int_{0}^{r}\int_{\partial\Omega} |u_t|^2 x_0^3 \,dx\,dt\,dx_0\\
&\quad +  C_1(\eta) (\|\mu_1\|_{C,2r} +\|\mu_2\|_{C,2r})\int_{\partial\Omega} \left[N^{r}(u)\right]^2 \,dx\,dt\\
&\quad +
          \eta\int_{0}^{r}\int_{\partial\Omega} |\nabla u|^2 x_0 \,dx\,dt\,dx_0,
\end{split}\end{equation}
for any $\eta>0$. 

The following lemma handles the area function in (\ref{Sq09}) in terms of  the square and non-tangential maximal functions.

\begin{lemma}\label{L:AS}
Let $u$ be a solution of \eqref{E:u} satisfying \eqref{E:UEB}, \eqref{E:B1}, \eqref{E:B2}, \eqref{E:A1}, and \eqref{E:A2} with Carleson norm bounded by $K$.
Then given $a>0$  there exists a constant $C= C(\Lambda, a,K)$ such that,
\[
A_{a}(u)(x,t)\le C S_{2a}(u)(x,t).
\]
From this we also have the global estimate
\[
\|A_{a}(u)\|_{L^{2}(\partial \Omega)}^{2} \leq C_2 \|S_{a}(u)\|_{L^{2}(\partial \Omega)}^{2}.
\]
\end{lemma}

\begin{proof}
Given the PDE \eqref{E:u} we have that
\[
|u_{t}|^{2} \leq 3|A|^2 |\nabla^{2}u|^{2} + 3\left(|\nabla A|^2 + |\boldsymbol{B}|^2\right) |\nabla u|^{2}.
\]
Therefore, from the definition of the area function, it follows
\[\begin{split}
& A_{a}^{2}(u)(x,t) \\
&= \int_{\Gamma_{a}(x,t)} |u_{t}|^{2} x_{0}^{-n+2} \,dx_0\,dy\,ds \lesssim \int_{0}^{\infty}x_{0}^{-n+3}\int_{Q_{x_0}} |u_{t}|^{2} \,dy\,ds\,dx_0 \\
&\leq 3 \int_{0}^{\infty}x_{0}^{-n+3}\int_{Q_{x_0}}  \left[|A|^2 |\nabla^{2}u|^{2} + \left(|\nabla A|^2 + |\boldsymbol{B}|^2\right) |\nabla u|^{2}\right] \,dy\,ds\,dx_0.
\end{split}\]
Here $$Q_{x_0}:=\{(y_0,y,s):\,|y_0-x_0|\le x_0/4\text{ and } |y-x|+|s-t|^{1/2}\le ax_0\}.$$

Hence for any fixed $y_0 >0$, we can use Lemma \ref{L:Caccio2} for $\nabla^{2} u$ (observe that the assumptions on the coefficients
in Lemma \ref{L:Caccio2} are satisfied on each $Q_{x_0}$). Also by the Carleson condition $|\nabla A|,|\boldsymbol{B}|\le K^{1/2}/x_0$ on $Q_{x_0}$, hence we obtain that
\[\begin{split}
& \int_{Q_{x_0}}  \left[|A|^2 |\nabla^{2}u|^{2} + \left(|\nabla A|^2 + |\boldsymbol{B}|^2\right) |\nabla u|^{2}\right] \,dy\,ds \\
& \leq \int_{Q_{2x_0}}  x_0^{-2} \left[C_a(K)|A|^2 |\nabla u|^{2} + 2K|\nabla u|^{2} \right] \,dy\,ds\\
&=C(\Lambda,a,K)x_0^{-2}\int_{Q_{2x_0}}|\nabla u|^{2} \,dy\,ds.
\end{split}\]
It follows that
\begin{equation}\label{AS02}\begin{split}
A_{a}^{2}(u)(x,t)&\le 3C(\Lambda,a,K) \int_{0}^\infty x_{0}^{-n+1}\int_{Q_{2x_0}}|\nabla u|^{2} \,dy\,ds\,dx_0\\
&\approx 3C(\Lambda,a,K)\int_{\Gamma_{2a}(x,t)} |\nabla u|^{2}x_0^{-n} \,dy_0\,dy\,ds.
\end{split}\end{equation}
As the last integral is just the square function (squared) the desired result holds. The global estimate follows from the local one.
\end{proof}

By Lemma \ref{L:AS} we see that the square function on the right-hand side of (\ref{Sq09}) is always preceded by $\eta>0$ which we are allowed to choose as required. We choose $\eta>0$ small enough so that all terms containing the square function are so small that they can be absorbed by the square function on the left-hand side.

This yields for some $C_3>0$:

\begin{equation}\label{Sq09a}\begin{split}
C_3&\|S^r(u)\|^2_{L^2(\partial\Omega)}\leq    \int_{\partial\Omega} \left(\partial_{x_0} u^2\right)(r,x,t) r  \,dx\,dt - \int_{\partial\Omega} u^2(r,x,t) \,dx\,dt\\
       &+ \int_{\partial\Omega} u^2(0,x,t) \,dx\,dt + K(\|\mu_1\|_{C,2r} +\|\mu_2\|_{C,2r}+\|\mu_2\|^{1/2}_{C,2r}) \int_{\partial\Omega} \left[N^{r}(u)\right]^2 \,dx\,dt.
\end{split}\end{equation}

We integrate the equation \eqref{Sq09a} in $r$ variable and average over $[0,r_0]$.
Because  $\left(\partial_{x_0} u^2\right) x_0=\partial_{x_0}\left(u^2 x_0\right) - u^2$, we see that \eqref{Sq09} becomes

\begin{equation}\label{E:sq10}\begin{split}
&C_3\int_{0}^{r_0}\int_{\partial\Omega} \left(x_0 - \frac{x_0^2}{r_0}\right) |\nabla u |^2 \,dx\,dt\,dx_0
+ \frac{2}{r_0} \int_{0}^{r_0}\int_{\partial\Omega} u^{2}(x_0,x,t)\,dx\,dt\,dx_0 \\
&\leq    \int_{\partial\Omega} u^2(r_0,x,t)\,dx\,dt
       + \int_{\partial\Omega} u^2(0,x,t) \,dx\,dt\\& + K(\|\mu_1\|_{C,2r} +\|\mu_2\|_{C,2r}+\|\mu_2\|^{1/2}_{C,2r})\int_{\partial\Omega} \left[N^{r_0}(u)\right]^2 \,dx\,dt.
\end{split}\end{equation}

Truncating the integral on the left-hand side to $[0,r_0/2]$ we finally obtain:

\begin{equation}\label{E:sq00}\begin{split}
&\frac{C_3}2\int_{0}^{r_0/2}\int_{\partial\Omega} |\nabla u |^2 x_0\,dx\,dt\,dx_0
+ \frac{2}{r_0} \int_{0}^{r_0}\int_{\partial\Omega} u^{2}(x_0,x,t)\,dx\,dt\,dx_0 \\
&\leq    \int_{\partial\Omega} u^2(r_0,x,t)\,dx\,dt
       + \int_{\partial\Omega} u^2(0,x,t) \,dx\,dt\\& + K(\|\mu_1\|_{C,2r} +\|\mu_2\|_{C,2r}+\|\mu_2\|^{1/2}_{C,2r}) \int_{\partial\Omega} \left[N^{r_0}(u)\right]^2 \,dx\,dt.
\end{split}\end{equation}
The local estimate for Lemma \ref{L:Square2} is obtained if we do not perform the sum over all coordinate patches, but instead use the estimate we have obtained for a single boundary cube $Q_r$. In this case the terms we denoted by $J_4$ do not cancel, but can be instead estimated in terms of the nontangential maximal function squared or as a product of the square and nontangential maximal functions. Both such terms can be handled leading to bound of the type \eqref{E:sq10aaa}. The Lipschitz constant $L$ makes appearance in \eqref{E:sq10aaa} via the flattening map \eqref{mapping}. The original surface ball $\Delta_r$ is mapped onto a subset of a surface ball $Q_r$ on a flat boundary. The Carlson norm of coefficients of a new PDE depends both of the Carleson norm of the original coefficients near $\Delta_r$ and the Lipschitz norm of the boundary $L$. Due to a deformation of balls by a factor of $L$ oscillation of coefficients might increase roughly by factor of $\sqrt{1+L^2}$. As the Carleson norm contains square of the oscillation it will increase up to a factor of $1+L^2$ as stated. \qed 
\vspace{1mm}

The following corollary is obtained from Lemma~\ref{L:Square} after further estimating the first integral on the right hand side of \eqref{E:sq00}.

\begin{corollary}\label{C:Square} Let $\Omega$ be as in Lemma~\ref{L:Square}.
Let $u$ be a nonnegative weak solution of \eqref{E:u}. For some small $r_0>0$ depending on the geometry of the domain $\Omega$, there exist constants $C_1,C_2>0$ such that
for $\epsilon=\|\mu_1\|_{C,2r} +\|\mu_2\|_{C,2r}+\|\mu_2\|^{1/2}_{C,2r}$
\begin{equation}\begin{split}\label{E:sq15bb}
\|S^{r_0/2}(u)\|^2_{L^2(\partial\Omega)}&=\int_{0}^{r_0/2} \int_{\partial\Omega} |\nabla u|^2 x_0 \,dx\,dt\,dx_0\\
&\leq    C_1\int_{\partial\Omega} u^2(0, x, t) \,dx\,dt
     + C_2\epsilon\int_{\mathbb{R}^{n}} \left[N^{r_0}(u)\right]^2 \,dx\,dt .\end{split}
\end{equation}
\end{corollary}

\begin{proof}
For any $ 1\leq p \leq \infty$, our goal is to show that for small $r>0$ and a nonnegative solution $u$
\begin{equation}\label{E:sq15}\begin{split}
 &\int_{\partial\Omega} u^p(r,x,t)\,dx\,dt \\
 &\leq \frac{2}{r}\int_{0}^{r}\int_{\partial\Omega} u^{p}(x_0,x,t)\,dx\,dt\,dx_0 + C_2\epsilon \int_{\partial\Omega} \left[N^{r}(u)\right]^p \,dx\,dt.
\end{split}\end{equation}
Clearly \eqref{E:sq00} and \eqref{E:sq15} gives us \eqref{E:sq15bb}.

When $p=\infty$, \eqref{E:sq15} holds by the maximum principle even with $\epsilon=0$. If \eqref{E:sq15} is true for $p=1$, then the interpolation argument yields \eqref{E:sq15} for any $1\le p\le\infty$.  Hence our goal is narrowed down to establish
\begin{equation}\begin{split}
 &\int_{\partial\Omega} u(r,x,t)\,dx\,dt \leq \frac{2+C_2\epsilon}{r}\int_{0}^{r}\int_{\partial\Omega} u(x_0,x,t)\,dx\,dt\,dx_0+ C_2\epsilon \int_{\partial\Omega} N^{r}(u) \,dx\,dt.
\end{split}\end{equation}
Consider a subsolution of $u$ that satisfies
\[
v_{t} = \di \left( A \nabla v \right) + \boldsymbol{B} \cdot \nabla v
\]
in the region $(\delta r, r) \times \partial\Omega$ that is strictly away from the boundary $\partial\Omega$. Here $\delta=\delta(\varepsilon) \in (0,1)$ will be determined later. We impose on $v$ the boundary conditions
$v = u$ on $\{r\} \times \partial\Omega$ and $v=0$ on $\{\delta r\} \times \partial\Omega$. If we are able to establish
\begin{equation}\label{E:sqv}
\int_{\partial\Omega} v(r,x,t) \,dt\,dx
\leq \frac{2}{(1-\delta)r} \int_{\delta r}^{r} \int_{\partial\Omega} v(x_0,x,t) \,dt\,dx\,dx_0+ \epsilon \int_{\partial\Omega} N^{r}(u) \,dx\,dt
\end{equation}
then the same inequality holds for $u$ as $v\le u$. Our conclusion will follow by choosing $\delta$ such that $2/(1-\delta) = 2+\epsilon$.


We construct a sequence of solutions $\{v_{m}\}_{m=-\infty}^{\infty}$ in two steps. Consider the usual cover of $\partial\Omega$ by a sequence
of parabolic boundary balls $Q(x_m,t_m,r)$ for some $(x_m,t_m)\in \partial\Omega$. As usual, we may assume that at most $K=K(n,N)>0$ such balls overlap. Let a nonnegative ${\widetilde v}_{m}$ solves the PDE
\[
\left({\widetilde v}_{m}\right)_{t} = \di \left( A \nabla {\widetilde v}_{m} \right) + \boldsymbol{B} \cdot \nabla {\widetilde v}_{m}.
\]
in $[\delta r,r]\times\partial\Omega$ with vanishing boundary data everywhere except on $\{r\}\times Q(x_m,t_m,r)$. Because the boundary balls $Q(x_m,t_m,r)$ cover $\partial\Omega$ we may arrange via partition of unity that ${\widetilde v}_{m}$ are nonnegative and supported on $\{r\}\times \partial\Omega$ such that
$$\sum_m \widetilde{v}_m = v = u,\qquad \text{ on $\{r\}\times \partial\Omega$}.$$

Hence, by the maximum principle it follows that
$$\sum_m \widetilde{v}_m = v \le u,\qquad \text{ on $[\delta r,r]\times \partial\Omega$}.$$

Next, let $0\le v_m\le \widetilde{v}_m$ be defined as follows. Let $k_1,\,k_2$ be positive integers to be defined later. We introduce new parabolic balls scaled by a factor $k_1$ in space and $k_2$ in time. Namely, 
for $r=r(k_1,k_2)>0$ small enough so that the parabolic boundary ball
$$Q_{k_1r,k_2r^2}(x_m,t_m):=\{(y,s)\in \partial\Omega:\,|x_m-y|\le k_1r\text{ and }|t_m-s|\le k_2 r^2\}$$
can localized to a single local coordinate chart let $v_m$ be a solution of the equation
\[
\left(v_{m}\right)_{t} = \di \left( A \nabla v_{m} \right) + \boldsymbol{B} \cdot \nabla v_{m}\quad\text{ in }(\delta r, r)\times Q_{k_1r,k_2r^2}(x_m,t_m)
\]
with vanishing initial and lateral boundary conditions on parabolic boundary of $(\delta r, r)\times Q_{k_1r,k_2r^2}(x_m,t_m)$ everywhere except on
$$v_m=\widetilde{v}_m \qquad\text{on }\{r\}\times Q_r(x_m,t_m).$$
By the maximum principle on $(\delta r, r)\times Q_{k_1r,k_2r^2}(x_m,t_m)$ we have $v_m\le \widetilde{v}_m$, hence if we extend $v_m$ by zero outside of this set
we have
$$v_m\le \widetilde{v}_m\qquad\text{ everywhere on }[\delta r,r]\times\partial\Omega$$.
It follows that
$$\sum_m {v}_m = v = u,\quad \text{ on $\{r\}\times \partial\Omega$}\quad\text{and}\quad \sum_m {v}_m \le v,\qquad \text{ on $[\delta r,r]\times \partial\Omega$}.$$

If we establish the inequality
\begin{equation}\label{E:sqvm}\begin{split}
& \int_{\partial\Omega} v_{m}(r,x,t) \,dt\,dx \\
&\leq \frac{2}{(1-\delta)r} \int_{\delta r}^{r} \int_{\partial\Omega} v_{m}(x_0,x,t) \,dt\,dx\,dx_0
+ \epsilon \int_{Q_r(x_m,t_m)} N (u)\,dt\,dx,
\end{split}\end{equation}
then \eqref{E:sqv} holds at it can be obtained by summing over all $m$.
The last term (with non-tangential maximal function) becomes $\epsilon K(n,N) \int_{\partial\Omega} N (u)\,dt\,dx$, where $K(n,N)$
is the maximum number of overlaps of parabolic balls $Q_r(x_m,t_m)$ at a single boundary point. This number is independent of $r$ and only depend on the geometry of $\partial\Omega$.

We shall consider \eqref{E:sqvm} in three ranges of $t$. Firstly, for
$t< t_m-r^2$ the solution $v_m$ vanishes. For any point $(r,y,s)$ with $(y,s)\in Q_r(x_m,t_m)$
we have a pointwise estimate
$$v_m(r,y,s)\le N^r(u)(y',s'), \text{ for all } (y',s')\in Q_{r/a}(y,s)$$
for a boundary parabolic ball $Q$ and $a>0$ being the aperture of the cones $\Gamma_a$. By averaging over $Q_{r/a}(y,s)$ then yields
$$\|v_m\|_{L^\infty(\{r\}\times Q_r(x_m,t_m))}\le \frac{C_a}{r^{n+1}}\int_{Q_r(x_m,t_m)}N^r(u)\,dx\,dt=:C_a\Phi_m.$$

This is a $L^\infty$ bound on the boundary data of $v_m$. It follows by the maximum principle that $0\le v_m\le \Phi_m$ everywhere.
At the time $t>t_m+r^2$ the solution $v_m$ will start decaying, due to vanishing boundary data at the whole lateral boundary.
Let us denote by ${\mathcal O}_\tau=[\delta r,r]\times \{|y-x_m|\le k_1r\}\times\{\tau\}$ (in local
coordinates on a coordinate chart containing $[\delta r,r]\times Q_{k_1r,k_2r^2}(x_m,t_m)$). Integration by parts  yields for $t>t_m+r^2$
$$\frac{d}{dt}\|v_m\|^2_{L^2({\mathcal O}_t)}\le -\lambda\|\nabla v_m\|^2_{L^2({\mathcal O}_t)}+\int_{{\mathcal O}_t}|\boldsymbol{B}||v_m||\nabla v_m|\,dX=I_1+I_2,$$
where the second term on the right-hand side can be further estimated by
\[\begin{split}
I_2\le&\frac{\lambda}{2} \int_{{\mathcal O}_t} \left| \nabla v_{m} \right|^{2} \,dX
    + \frac{2}{\lambda} \int_{{\mathcal O}_t} |\boldsymbol{B}|^{2} |v_{m}|^{2} \,dX \\
&\leq \frac{\lambda}{2} \|\nabla v_m\|^2_{L^2({\mathcal O}_t)}
    + \frac{2}{(\delta r)^{2}\lambda} \int_{{\mathcal O}_t} \left( x_0 |\boldsymbol{B}| \right)^{2} |v_{m}|^{2} \,dX
\end{split}\]
because $x_0 \in (\delta r, r)$. We now we apply the Poincar\'e inequality, Lemma~\ref{L:PIg}
\[-\frac{\lambda}{2}\|\nabla v_m\|^2_{L^2({\mathcal O}_t)}\le -\frac{c(n,\lambda)}{r^2}\| v_m\|^2_{L^2({\mathcal O}_t)}.
\]
Hence it follows that
\[
\frac{d}{dt}\|v_m\|^2_{L^2({\mathcal O}_t)}\le \frac1{r^2}\left[-c(n,\lambda)+\frac{2\|\mu_1\|_{C,r}}{\delta^{2}\lambda}\right]\| v_m\|^2_{L^2({\mathcal O}_t)}
\]
For $\|\mu_1\|_{C,r}$ is sufficiently small so that $\frac{2\|\mu_1\|_{C,r}}{\delta^{2}\lambda}\le\frac{c(n,\lambda)}2$ we get by the Gr\"onwall's inequality \[
\|v_m\|^2_{L^2({\mathcal O}_t)}\leq \exp \left(-\frac{c(n,\lambda)(t-t_m-r^2)}{2r^2}\right) \|v_m\|^2_{L^2({\mathcal O}_{t^2})}.
\]
Using the $L^2-L^\infty$ smoothing (c.f. Lemma \ref{L:Caccio} implying that the $L^\infty$ norm of the solution at later times can be estimated using the $L^2$ norm at the earlier time) we will have for all $t\ge t_m+2r^2$
\begin{equation}\begin{split}
\|v_m\|^2_{L^\infty({\mathcal O}_t)}\le& \frac{C}{k_1^{n-1}r^n}\|v_m\|^2_{L^2({\mathcal O}_{t-r^2})}\\\leq& \frac{C}{k_1^{n-1}r^n}\exp \left(-\frac{c(n,\lambda)(t-t_m-2r^2)}{2r^2}\right) \|v_m\|^2_{L^2({\mathcal O}_{t_m+r^2})}\\
\le&\frac{C'}{r^{2n+2}}\exp \left(-\frac{c(n,\lambda)(t-t_m-2r^2)}{2r^2}\right) \left(\int_{Q_r(x_m,t_m)} N (u)\,dt\,dx\right)^2.
\end{split}\end{equation}

It follows that for any $\epsilon'>0$ (to be determined later) we can pick $k_2$ such that
$$C'\exp \left(-\frac{c(n,\lambda)(k_2+2)}{2}\right)<(\epsilon')^2,$$
then for all $t\ge t_m+k_2r^2$

\begin{equation}\label{eq:upper}
\|v_m\|_{L^\infty({\mathcal O}_t)}\le \frac{\epsilon'}{r^{n+1}} \int_{Q_r(x_m,t_m)} N (u)\,dt\,dx=\epsilon'\Phi_m.
\end{equation}

It follow that for $t\le t_m-r^2$ the solution $v_m$ vanishes and for $t\ge t_m+k_2r^2$ the solution is very small. It is therefore sufficient to focus on $t_m-r^2\le t\le t_m+k_2r^2$ and prove that \eqref{E:sqvm} must hold there with all integrals restricted to this time interval.


We would like to compare the solution $v_m$ with a solution of a constant coefficient PDE $w_m$
\[
\left( {w}_{m} \right)_{t} = \di \left( \tilde{A} \nabla {w}_{m} \right)
\]
in $(\delta r, r)\times   Q_{k_1r,k_2r^2}(x_m,t_m) $ that shares the boundary data with $v_m$.  We pick $\tilde{A}$ to be the average of the matrix $A$ over the box
$(\delta r, r)\times Q_{k_1r,k_2r^2}(x_m,t_m)$.
Clearly, ${w}_{m}= 0$ if $t < t_m-r^2$ and \eqref{eq:upper} holds for $w_m$ as well.
Let
\[
\tilde{w}_{m} (X) := \int_{t_m-r^2}^{t_m+k_2r^2} {w}_{m}(X, t) \,dt
\]
which solves the elliptic differential equation
\[
0\le{w}_{m} (\cdot, t_m+k_2r^2) = \di \left( \tilde{A} \nabla \tilde{w}_{m} \right).
\]
Because  ${w}_{m} (\cdot, t_m+k_2r^2) \leq \epsilon' \Phi_{m}$, we consider
\[
{z}_{m} (x_0, x) = \tilde{w}_{m} (x_0, x) - \frac{\epsilon' \Phi_{m}}{2\widetilde{a_{00}}} [(x_0 - (1+\delta) r/2)^{2}-((1-\delta)r/2)^2]\ge 0.
\]
Note that this guarantees that $z_m(\delta r,x)=\tilde{w}_{m} (\delta r, x)$ and $z_m(r,x)=\tilde{w}_{m} (r, x)$. Also,
$\di \left( \tilde{A} \nabla z_{m} \right)=\di \left( \tilde{A} \nabla \tilde{w}_{m} \right)-\epsilon' \Phi_{m}={w}_{m} (\cdot, t_m+k_2r^2)-\epsilon' \Phi_{m}\le0$, and hence $z_m$ is a super-solution of an elliptic PDE with the same boundary data as $\tilde{w}_{m}$.
The mean value property of nonnegative solutions for this PDE has been studied in \cite{DPP}. It has been established there that the following integral inequality holds
\[
\int_{B_{r}(x_m)} z_{m}(r,\cdot) \,dx
\leq \frac{2+C(k_1)}{(1-\delta)r} \int_{\delta r}^{r}\int_{B_{k_{1} r}(x_m)} z_{m} \,dx\,dx_0,
\]
provided  $z_m$ is a solution. Here $C(k_1)\to 0+$ for large $k_1$. Out $z_m$ is not a solution but a super-solution but it is easy to observe that if the estimate above holds for solutions it is also true for
super-solutions as by the comparison principle for elliptic PDEs for any super-solution $z_m$ we can find a solution with same left-hand side but smaller right-hand side in the estimate above. In particular, the estimate above therefore holds for our function $z_m$.

We make a choice of $k_1$ large enough so that $C(k_1)/(1-\delta)\le\epsilon$. Recall that we have chosen
$\delta$ earlier such that $2/(1-\delta)\le 2+\epsilon$. If follows that $\frac{2+C(k_1)}{(1-\delta)r}\le\frac{2+2\epsilon}{r}$.

We apply this for our function $z_m$. If follows that
\[\begin{split}
&\int_{B_{r}(x_m)} \tilde{w}_{m}(r,\cdot) \,dx=\int_{B_{r}(x_m)} z_{m}(r,\cdot) \,dx\le\\
&\leq \frac{2+2\epsilon}{r} \int_{\delta r}^{r}\int_{B_{k_{1} r(x_m)}} \tilde{w}_m \,dx\,dx_0 +
|B_{k_1r}(x_m)|\frac{\epsilon'\Phi_m r^2}{4\widetilde{a_{00}}},
\end{split}\]
where the last term is a (fairly) crude estimate of the contribution of the term $- \frac{\epsilon' \Phi_{m}}{2\widetilde{a_{00}}} [(x_0 - (1+\delta) r/2)^{2}-((1-\delta)r/2)^2]$ that we subtracted off $\tilde{w}_m$. Recall, that we have made a conditional choice of $k_2$ (depending on $\epsilon'$)
but we have not specified $\epsilon'$. We fix this now and choose $\epsilon'=\epsilon4\widetilde{a_{00}}(k_1)^{-n+1}$ which implies that
$$|B_{k_1r}(x_m)|\frac{\epsilon'\Phi_m r^2}{4\widetilde{a_{00}}}\le \epsilon\int_{Q_r(x_m,t_m)}N^r(u)\,dx\,dt.$$
We now go back the $w_m$ and deduce the following inequality
\begin{equation}\label{CCineq}
\begin{split}
&\int_{B_{r}(x_m)\times[t^1,t^3]} {w}_{m}(r,\cdot) \,dx\,dt\le\\
&\leq \frac{2+2\epsilon}{r} \int_{\delta r}^{r}\int_{B_{k_{1} r(x_m)}\times[t_m-r^2,t_m+k_2r^2]} {w}_m \,dx\,dt\,dx_0 +\epsilon\int_{Q_r(x_m,t_m)}N^r(u)\,dx\,dt.
\end{split}
\end{equation}

What remains to be done is to estimate the difference $|w_m-v_m|$ on $[\delta r,r]\times B_{k_{1} r(x_m)}\times[t_m-r^2,t_m+k_2r^2]$ in a norm $L^1$ or any stronger
norm. If we establish
\begin{equation}\label{CCineq2}
\frac1r\|w_m-v_m\|_{L^1([\delta r,r]\times B_{k_{1} r(x_m)}\times[t_m-r^2,t_m+k_2r^2])}\le \frac{\epsilon}{k_1^{n-1}k_2}\int_{B_{k_{1} r}(x_m)\times [t_m-r^2,t_m+k_2r^2]}N^r(u)\,dx\,dt,
\end{equation}
then we obtain from \eqref{CCineq}
\begin{equation}
\begin{split}
&\int_{B_{r}(x_m)\times[t_m-r^2,t_m+k_2r^2]} {v}_{m}(r,\cdot) \,dx\,dt
\leq \frac{2+2\epsilon}{r} \int_{\delta r}^{r}\int_{B_{k_{1} r(x_m)}\times[t_m-r^2,t_m+k_2r^2]} {v}_m \,dx\,dt\,dx_0\\ + &\epsilon\int_{Q_r(x_m,t_m)}N^r(u)\,dx\,dt+\frac{3\epsilon}{k_1^{n-1}k_2}\int_{B_{k_{1} r}(x_m)\times [t_m-r^2,t_m+k_2r^2]}N^r(u)\,dx\,dt,
\end{split}
\end{equation}
which is what we want.  Since $B_{k_{1} r}(x_m)\times [t_m-r^2,t_m+k_2r^2]$ is the stretch of $Q_r(x_m,t_m)$ by the factor of $k_1$ in the spatial variables and factor $k_2$ in the time direction so $B_{k_{1} r}(x_m)\times [t_m-r^2,t_m+k_2r^2]$ is expected to have overlap with approximately $Ck_1^{n-1}k_2$ original Carleson boxes $Q_r(x_j,t_j)$, $j\in\Z$. That means that summing
$$\frac{\epsilon}{k_1^{n-1}k_2}\int_{B_{k_{1} r}(x_m)\times [t_m-r^2,t_m+k_2r^2]}N^r(u)\,dx\,dt,$$
over all $m$ will produce an error term of order $\epsilon\int_{\partial\Omega}N^r(u)\,dx\,dt$. We get the same
error term summing over 
$$\epsilon\int_{Q_r(x_m,t_m)}N^r(u)\,dx\,dt.$$
\vglue2mm

Let us now proceed with the estimate of $\|w_m-v_m\|$. We just use the standard $L^2$ theory. Consider $z_m=w_m-v_m$. Then $z_m$ solves the PDE

$$\left(z_{m}\right)_{t} = \di \left( \tilde{A} \nabla z_{m} \right) - \boldsymbol{B} \cdot \nabla v_{m}+ \di\left((\tilde{A}-A)\nabla v_m\right)$$
on $[\delta r,r]\times B_{k_{1} r}(x_m)\times [t_m-r^2,t_m+k_2r^2]$ with vanishing initial and lateral boundary data (since $v_m$ and $w_m$ coincide there).
Hence we can multiply both sides the the equation by $z_m$ and integrate in space yielding
\[\begin{split}\frac{d}{dt}\|z_m\|^2_{L^2({\mathcal O}_t)}&\le -\lambda \|\nabla z_m\|^2_{L^2({\mathcal O}_t)}\\&+\frac{1}{\delta}\int_{{\mathcal O}_t}x_0|\boldsymbol{B}|\frac{|z_m|}{r}|\nabla v_m|dx\,dx_0+\int_{{\mathcal O}_t}|\tilde{A}-A||\nabla v_m||\nabla z_m|dx\,dx_0\end{split}\]
for all $t_m-r^2<t<t_m+k_2r^2$ using the ellipticity condition, integration by parts and the fact that $\frac{x_0}r\ge \delta$ on $[\delta r,r]$.
Recall the notation ${\mathcal O}_t$ we introduced above, which denotes the time slice of our domain in time $t$.

Using \eqref{E:B1}, \eqref{E:A1} and the Poincar\'e inequality (Lemma \ref{L:PIg})  we obtain
\begin{equation}
\begin{split}
\frac{d}{dt}\|z_m\|^2_{L^2({\mathcal O}_t)}&\le -\lambda \|\nabla z_m\|^2_{L^2({\mathcal O}_t)}+C\frac{\|\mu_1\|^{1/2}_{C,r}}{\delta}\|\nabla v_m\|_{L^2({\mathcal O}_t)}\|\nabla z_m\|_{L^2({\mathcal O}_t)}\\
&+\frac{\max\{k_1,k_2\}\|\mu_2\|^{1/2}_{C,r}}{\delta}\|\nabla v_m\|_{L^2({\mathcal O}_t)}\|\nabla z_m\|_{L^2({\mathcal O}_t)}.
\end{split}
\end{equation}
We eliminate the term $-\lambda \|\nabla z_m\|^2_{L^2({\mathcal O}_t)}$ by using Cauchy-Schwarz on the other two terms
$$\frac{d}{dt}\|z_m\|^2_{L^2({\mathcal O}_t)}\le \left[\frac{C^2\|\mu_1\|_{C,r}+\max\{k^2_1,k^2_2\}\|\mu_2\|_{C,r}}{\lambda\delta^2}\right]\|\nabla v_m\|^2_{L^2({\mathcal O}_t)}.
$$
Since $\|z_m\|^2_{L^2({\mathcal O}_{t_1})}=0$ it follows that for $t<t_m+k_2r^2$
$$\|z_m\|^2_{L^2({\mathcal O}_t)}\le \left[\frac{C^2\|\mu_1\|_{C,r}+\max\{k^2_1,k^2_2\}\|\mu_2\|_{C,r}}{\lambda\delta^2}\right]\int_{t^1}^{t^3}\|\nabla v_m\|^2_{L^2({\mathcal O}_t)}\,dt$$
and hence
\begin{equation}
\begin{split}
&\|z_m\|^2_{L^2([\delta r,r]\times B_{k_{1} r}(x_m)\times [t_m-r^2,t_m+k_2r^2])}\\&\le \left[\frac{C^2\|\mu_1\|_{C,r}+\max\{k^2_1,k^2_2\}\|\mu_2\|_{C,r}}{\lambda\delta^2}\right]k_2r^2\|\nabla v_m\|^2_{L^2([\delta r,r]\times B_{k_{1} r}(x_m)\times [t_m-r^2,t_m+k_2r^2])}.
\end{split}
\end{equation}
The norm $\|\nabla v_m\|^2_{L^2([\delta r,r]\times B_{k_{1} r}(x_m)\times [t_m-r^2,t_m+k_2r^2])}$ can be estimated using Cacciopoli inequality (Lemma \ref{L:Caccio}) by
$$\frac{C_\delta}{r^2}\|v_m\|^2_{L^2([\delta r ,r]\times B_{(k_{1}+\delta) r}(x_m)\times [t_m-(1+\delta) r^2,t_m+(k_2+\delta) r^2])},$$
i.e on a slightly enlarged Carleson box. Here we are enlarging the region in some directions but we do not have to enlarge the interval $[\delta r,r]$ since $v_m$ vanishes on two lateral boundaries when $x_0=\delta r$ and $x_0=r$. Hence we can use an odd reflection across these two boundaries to obtain a boundary version of Cacciopoli.

 This quantity we further estimate using the non-tangential maximal function $N(v_m)\le N(u)$ giving us
\begin{equation}
\begin{split}
&\|z_m\|^2_{L^2([\delta r,r]\times B_{k_{1} r}(x_m)\times [t_m-r^2,t_m+k_2r^2])}\le \left[\frac{C^2\|\mu_1\|_{C,r}+\max\{k^2_1,k^2_2\}\|\mu_2\|_{C,r}}{\lambda\delta^2}\right]
k_2C_\delta\times\\&\times\left|[\delta r ,r]\times B_{(k_{1}+\delta) r}(x_m)\times [t_m-(1+\delta) r^2,t_m+(k_2+\delta) r^2]\right|\times\\&\times\frac{C_a^2}{(\delta r)^{2n+2}}\left(\int_{B_{k_{1} r}(x_m)\times [t_m-r^2,t_m+k_2r^2]}N^r(u)\,dx\,dt\right)^{2}.
\end{split}
\end{equation}
Hence
\begin{equation}\label{eq:estim}
\begin{split}
&\|z_m\|_{L^2([\delta r,r]\times B_{k_{1} r}(x_m)\times [t_m-r^2,t_m+k_2r^2])}\le \\& \left[C^2\frac{\|\mu_1\|_{C,r}+\max\{k^2_1,k^2_2\}\|\mu_2\|_{C,r}}{\lambda\delta^2}k_2C_\delta(2k_1)^{n-1}2k_2r^{n+2}\right]^{1/2}\times\\&\times
\frac{C_a}{(\delta r)^{n+1}}\int_{B_{k_{1} r}(x_m)\times [t_m-r^2,t_m+k_2r^2]}N^r(u)\,dx\,dt\\&=\frac{C(\mu_1,\mu_2,k_1,k_2,\delta,\lambda,a)}{r^{n/2}}\int_{B_{k_{1} r}(x_m)\times [t_m-r^2,t_m+k_2r^2]}N^r(u)\,dx\,dt.
\end{split}
\end{equation}
Hence for the $L^1$ norm we have
\begin{equation}\label{eq:estim2}
\begin{split}
&\|z_m\|_{L^1([\delta r,r]\times B_{k_{1} r}(x_m)\times [t_m-r^2,t_m+k_2r^2])} \\ \le&
\left|[\delta r,r]\times B_{k_{1} r}(x_m)\times [t_m-r^2,t_m+k_2r^2]\right|^{1/2}\times \|z_m\|_{L^2([\delta r,r]\times B_{k_{1} r}(x_m)\times [t_m-r^2,t_m+k_2r^2])}\\
\le &C(\mu_1,\mu_2,k_1,k_2,\delta,\lambda,a)r\int_{B_{k_{1} r}(x_m)\times [t_m-r^2,t_m+k_2r^2]}N^r(u)\,dx\,dt,
\end{split}
\end{equation}
which is the desired estimate. We have to assume Carleson condition on the coefficients $A$, $B$ small enough so that
$$C(\mu_1,\mu_2,k_1,k_2,\delta,\lambda,a)\le\frac{\varepsilon}{k_1^{n-1}k_2}.$$

\end{proof}

\section{Comparability of the non-tangential maximal function and the square function}\label{S3:Comp}

The results of the previous section, namely Lemma \ref{L:Square} immediately imply that
$$\|S^{r/2}(u)\|_{L^2(\partial\Omega)}\le C\|N^{r}(u)\|_{L^2(\partial\Omega)},$$
for any solution $u$ of the parabolic PDE whose coefficients satisfy the Carleson condition with $C>0$ independent of $u$.

We want to establish that the reverse estimate is also true. Our goal is significantly simplified by the following local estimate from \cite{Rn} Theorem 1.3.

\begin{lemma}\label{L:NL}
Let $u$ be a solution on $U$ of \eqref{E:u} whose coefficients satisfy the Carleson conditions \eqref{E:B1}-\eqref{E:A2} on all parabolic balls of size $\le r_0$. Then there exists a constant $C$ such that for any $r\in (0,r_0/8)$,
\begin{equation}\label{E:NL2}
\int_{Q_{r}} \left[N_{a/12} (u)\right]^2 \,dx\,dt
\leq C \left[ \int_{Q_{2r}} \left[A_a (u)\right]^2 \,dx\,dt + \int_{Q_{2r}} \left[S_a (u)\right]^2 \,dx\,dt \right] + Cr^{n+1}|u(A_{Q_r})|^2.
\end{equation}
Here $A_{Q_r}$ be so-called corkscrew point relative to cube $Q_r$ (that is a point inside $U$ of whose distance to the boundary $\partial U$
and $Q_r$ is approximately $r$).
\end{lemma}

\noindent{\it Remark.} Theorem 1.3 of \cite{Rn} is stated using a different last term on the right-hand side, however by looking into the details of the proof c.f. \cite[Proposition 5.3]{Rn} we see that we can use $Cr^{n+1}|u(A_{Q_r})|^2$ there. Also \cite{Rn} states all results for symmetric operators without any drift term.
However, the proof is a standard stoping time argument adapted from the elliptic setting and the proof of this particular theorem uses the PDE which $u$ is a solution of in one place only. An extra drift term does not cause any issue there. \vglue2mm

Based on this $L^2$ estimates of the non-tangential maximal function we obtain the following global version of the Lemma \ref{L:NL}.

\begin{theorem}\label{T:NG}
Let $u$ be a solution of the equation $u_t-\di(A\nabla u)=\boldsymbol{B}\cdot \nabla u$ in a domain $\Omega$ as in the definition \ref{D:domain} of character $(L,N,C_0)$. Assume that
the matrix $A$ is uniformly elliptic and the vector $\boldsymbol{B}$ is bounded on $\Omega$ and its coefficients satisfy \eqref{carlIII} and \eqref{carlB} with bounded Carleson norm. Then there exists a constant $C$ such that
\[
\int_{\partial \Omega} [N (u)]^2 \,dx\,dt \leq C \left[\int_{\partial \Omega} [S (u)]^2 \,dx\,dt + \int_{\partial \Omega} u^{2}(0, \cdot) \,dx\,dt\right] .
\]
\end{theorem}

\begin{proof}
We begin with a local inequality based on \eqref{E:NL2}. In the subspace
\[
\mathcal{S} = \left\{ u : \int_{Q_{r}} u \,dx\,dt = 0 \right\},
\]
we wish to show that for some constant $C$
\begin{equation}\label{T3.5E01}
\int_{Q_{r}} [N_{a/12}(u)]^2 \,dx\,dt \leq C \int_{Q_{2r}} [S_a(u)]^2 \,dx\,dt+C\int_{Q_{2r}} [A_a(u)]^2 \,dx\,dt .
\end{equation}
We proceed by contradiction. If \eqref{T3.5E01} fails, then for arbitrary large $C$ there exists $u$ such that
\[
\int_{Q_{r}} [N_{a/12}(u)]^2 \,dx\,dt > C \left[\int_{Q_{2r}} [S_a(u)]^2 \,dx\,dt+\int_{Q_{2r}} [A_a(u)]^2\,dx\,dt\right] .
\]
Therefore we can find a sequence of solutions $\{u_{k}\}_{k=1}^{\infty}$ satisfying
\begin{subequations}\label{T3.5E02}
\begin{equation}\label{T3.5E02a}
\int_{Q_{r}} [N_{a/12} (u_{k})]^2 \,dx\,dt = 1,
\end{equation}
\begin{equation}\label{T3.5E02b}
\int_{Q_{2r}} [S_a (u_{k})]^2 \,dx\,dt \leq \frac{1}{k},\quad \int_{Q_{2r}} [A_a (u_{k})]^2 \,dx\,dt \leq \frac{1}{k},
\end{equation}
\begin{equation}\label{T3.5E02c}
\int_{Q_{r}} u_{k} \,dx\,dt = 0 .
\end{equation}
\end{subequations}

Because of \eqref{T3.5E02a}, for any interior point $(y_0, y, s)\in \Gamma_{a/12} (x,t)$ where $(x,t)\in Q_{r}$, we have
\[
\left|u_{k} (y_0, y, s)\right| \leq C .
\]
for some constant $C>0$ ($C$ depends on the distance $y_0$ to the boundary and blows up as $y_0\to 0+$).
By Arzela-Ascoli theorem, we can find a subsequence $\{u_{k_{j}}\}_{j=1}^{\infty}$ that converges locally uniformly to $u$,  on all compact subsets $K$ of the union of the cones $\Gamma_{a/2} (x,t)$ for $(x,t)\in Q_{2r}$.

Moreover, on such $K$ we have by the estimates for the square and area functions the convergence of the full gradient $Du_k$ to zero, i.e., $\|D u_k\|_{L^2(K)}\to 0$. It follows that $u_k$ has to converge to a function $u$ with $D u=0$ on $K$, hence $u$ is constant on the union of all non-tangential cones $\Gamma_a(x,t)$ where $(x,t)\in Q_{2r}$.


Because $\{u - u_{k_j}\}_{j=1}^{\infty}$ is a sequence of weak solutions, the Lemma~\ref{L:NL} applies
\begin{equation}\begin{split}
& \int_{Q_{r}} [N (u- u_{k_j})]^2 \,dx\,dt \\
&\leq C \left[\int_{Q_{2r}} [S (u- u_{k_j})]^2 \,dx\,dt+\int_{Q_{2r}} [A (u- u_{k_j}) ]^2\,dx\,dt + r^{n+1} (u-u_{k_j})(A_{Q_r})\right] \\
&= C \left[\int_{Q_{2r}} [S (u_{k_j})]^2 \,dx\,dt+\int_{Q_{2r}} [A (u_{k_j}) ]^2\,dx\,dt + r^{n+1} (u-u_{k_j})(A_{Q_r})\right] \\
&\to 0,
\end{split}\end{equation}
by the fact that $u-u_{k_j}\to 0$ at $A_{Q_r}$. Since
$$\|(u- u_{k_j})\|_{L^1(Q_r)}\le C(r)\|(u- u_{k_j})\|_{L^2(Q_r)}\le C(r)\|N(u- u_{k_j})\|_{L^2(Q_r)}\to 0,$$
and the functions $u_{k_j}$ have zero mean on $Q_r$ it follows that $u$ has zero mean as well. As $u$ is constant we get that $u=0$ everywhere.

On the other hand
\begin{equation}\begin{split}
& \int_{Q_{2}} [N (u)]^2 \,dx\,dt = \int_{Q_{2}} \left[\sup_{\Gamma_{a}} \left| u_{k_{j}} - (u_{k_{j}} - u)\right|\right]^2 \,dx\,dt\\
&\geq  \int_{Q_{r}} [N (u_{k_j})]^2 \,dx\,dt - \int_{Q_{r}} [N \left(u_{k_j} - u \right)]^2 \,dx\,dt\to 1,
\end{split}\end{equation}
which contradicts the fact that $N(u)=0$ as $u=0$. Therefore on the subspace $\mathcal{S}$, \eqref{T3.5E01} holds.

For a general $u$, clearly $v=[u-|Q_r|^{-1}\int_{Q_r}u\,dx\,dt]\in\mathcal{S}$ and hence \eqref{T3.5E01} applies to $v$.
This gives
\begin{equation}\begin{split}\label{T3.5E04}
&\int_{Q_{r}} [N (u)]^2 \,dx\,dt \leq\\ &C \left[\int_{Q_{2r}} [S (u)]^2 \,dx\,dt +\int_{Q_{2r}} [A (u)]^2 \,dx\,dt + \left(\int_{Q_{r}} u (0, x, t) \,dx\,dt\right)^{2}\right].\end{split}
\end{equation}
Hence, if we use the Cauchy-Schwarz inequality on the last term and then sum over all parabolic balls $Q_r$ covering $\partial\Omega$  we obtain the global estimate we aimed for (by Lemma \ref{L:AS}).
\end{proof}

\section{$L^2$ solvability.}

\begin{lemma}\label{L:Main} Let $\Omega$ be an admissible domain with character $(L,N,C_0)$ where $L$ is small and $\mathcal{L}u = u_{t} - \di \left( A \nabla u\right) - \boldsymbol{B}\cdot \nabla u$ be a parabolic operator whose matrix satisfies the uniform ellipticity and boundedness for constants $\lambda$ and $\Lambda$ and either
\begin{equation}\label{E:carl6}\begin{split}
d\mu_1 &=\left[\delta(X,t)\left( \sup_{B_{\delta(X,t)/2}(X,t)} |\nabla A| \right)^2+\delta^3(X,t)\left( \sup_{B_{\delta(X,t)/2}(X,t)} |\partial_t A| \right)^2 \right.\\
& \left. +\delta(X,t)\left( \sup_{B_{\delta(X,t)/2}(X,t)} |\boldsymbol{B}| \right)^2\right] \,dX\,dt
\end{split}\end{equation}
or
\begin{equation}\label{E:carl4aa}
d\mu_2=\left(\delta(X,t) |\nabla A|^2+ \delta^3(X,t) |\partial_t A|^2 + \delta(X,t) |\boldsymbol{B}|^2\right) \,dX\,dt
\end{equation}
is density of a small Carleson measure on all Carleson regions of size $\le r_0$. In addition in the case \eqref{E:carl4aa} holds we also assume that
\begin{equation}\label{E:carl4bb}
\delta(X,t) |\nabla A|+\delta^2(X,t) |\partial_t A| + \delta(X,t) |\boldsymbol{B}|\le C,
\end{equation}
for a small constant $C$. Then the Dirichlet problem $\mathcal{L} u = 0$ with data in $L^{2}(\partial\Omega, d\sigma)$ is solvable. Furthermore, for every $f\in L^{2}(\partial\Omega, d\sigma)$, the weak solution $u$ to the parabolic operator $\mathcal{L} u = 0$ satisfies the estimate
\[
\|N (u)\|_{L^{2}(\partial\Omega, d\sigma)} \leq C \|f\|_{L^{2}(\partial\Omega, d\sigma)}
\]
for some constant $C$ depending only on the constants characterizing the domain $\Omega$ and the boundedness and ellipticity of the matrix $A$.
\end{lemma}

\begin{proof} Note that we may assume that $\Omega$ in addition to satisfying Definition $\ref{D:domain}$ also has a smooth boundary. This is due to the subsection \ref{S:alld} where we have established existence of a $C^\infty$ diffeomorphism $f^\epsilon:\Omega\to\Omega_\varepsilon$, which allows us to consider our parabolic PDE on a smooth domain $\Omega_\epsilon$ instead of $\Omega$. The new equation on $\Omega_\epsilon$ will have coefficients of small Carleson norm, if the original coefficients and the constant $L$ (from the character of the domain) are assumed to be small. Note also, that there is no issue with a further pull-back of our PDE onto the upper half-space $U$, since the composition $(f^\epsilon)^{-1}\circ\rho:U\to\Omega$  (where $\rho:U\to\Omega_\epsilon$) is a map of the type we considered in the subsection \ref{S11:PC}.

Consider $f^+=\max\{0,f\}$ and $f^-=\max\{0,-f\}$ where $f\in C_0(\partial\Omega)$ and denote the corresponding solutions with these boundary data $u^+$ and $u^-$, respectively.  Hence we may apply the Corollary~\ref{C:Square} separately to $u^+$ and $u^-$. By the maximum principle, these two solutions are nonnegative. It follows that for any such nonnegative $u$ we have
\[
 \|S^{r}(u)\|^{2}_{L^{2} (\partial\Omega)} \leq C \|f\|^{2}_{L^2 (\partial\Omega)} + C (\|\mu\|^{1/2}_{C}+\|\mu\|_{C}) \|N^{2r}(u)\|^{2}_{L^{2}(\partial\Omega)}
\]
and Theorem~\ref{T:NG}
\[
\|N^{r}(u)\|^{2}_{L^{2}(\partial\Omega)} \leq C\|f\|^{2}_{L^2 (\partial\Omega)} + C \|S^{2r}(u)\|^{2}_{L^2(\partial\Omega)}.
\]
Here $\|\mu\|_{C}$ is the Carleson norm of $(\ref{E:carl6})$ on Carleson regions of size $\le r_0$. Since we are assuming $\|\mu\|_{C}$ is small, clearly we have $\|\mu\|_{C}\le C\|\mu\|^{1/2}_{C}$.
By rearranging these two inequalities, we obtain, for $0<r\le r_0/8$,
\[
\|N^{r}(u)\|^{2}_{L^{2}(\partial\Omega)} \leq C\|f\|^{2}_{L^2 (\partial\Omega)} + C \|\mu\|^{1/2}_{C} \|N^{4r}(u)\|^{2}_{L^{2}(\partial\Omega)}.
\]
Here $N^h$ denotes the truncation at height $h$. If for some constant $M>0$, if we prove
\begin{equation}\label{E:NR}
\|N^{4r}(u)\|^{2}_{L^{2}(\mathbb{R}^{n})} \leq M \|N^{r}(u)\|^{2}_{L^{2}(\mathbb{R}^{n})},
\end{equation}
then for $\|\mu\|_{C}$ small (less than $1/(CM)^2$), we obtain \eqref{E:estim3}.

We first make an observation that for any, $(y_0, y, s) \in \Gamma_{a}^{4r} (x,t)$, there exists a point $(z_0, z, \tau^{*}) \in \Gamma_{8a}^{r}(x,t)$ such that $\tau^{*} > s + r^2$. Hence by Lemma~\ref{Harnack} (Harnack inequality), there exists an a priori constant $M$ such that
\[
u(y_0, y, s) \leq M u(z_0, z, \tau^{*}) .
\]
Therefore, we obtain
\[
N_{a}^{4r} (u) \leq M u(z_0, z, \tau^{*}) \leq M N_{8a}^{r} (u).
\]
Hence, if we establish that the non-tangential maximal functions $N_{8a}^{r} (u)$ and $N_{a}^{r} (u)$ defined using cones of different aperture are equivalent,  then we are done. The equivalence of these norms is established in Lemma \ref{L:Equ} above. The result for $N(u)$ follows, by combining estimates for $N(u^+)$ and $N(u^-)$. A this point we can use the non-truncated version of nontangential maximal function to state our results since our domain $\Omega$ is admissible and hence bounded in space (not time). But this gives that for sufficiently large $d=\sup_{\tau\in\mathbb R}{\mbox{diam}(\Omega\cap \{t=\tau\})}$ we have
$$N(u)=N^d(u).$$ Iterating \eqref{E:NR} then gives for arbitrary small $r>0$
$$\|N(u)\|^2_{L^2(\partial\Omega)}=\|N^{d}(u)\|^{2}_{L^{2}(\partial\Omega)} \leq C(r,d) \|N^{r}(u)\|^{2}_{L^{2}(\partial\Omega)}.$$
\end{proof}

\end{document}